\documentclass{amsart}

\usepackage[latin1]{inputenc}
\usepackage{amsmath, amsthm, amssymb, booktabs, multirow, graphicx, booktabs}
\usepackage[colorlinks]{hyperref}
\usepackage{todonotes}
\usepackage{array}
\usepackage{subcaption}

\newtheorem{defin}{Definition}[section]

\newtheorem{theorem}[defin]{Theorem}
\newtheorem{lemma}[defin]{Lemma}
\newtheorem{corollary}[defin]{Corollary}

\newtheorem{remark}[defin]{Remark}

\newcommand{\R}{\mathbb{R}}

\newcommand{\Z}{\mathbb{Z}}

\newcommand{\lset}{\left\{}
\newcommand{\rset}{\right\}}

\DeclareMathOperator{\id}{id}
\DeclareMathOperator{\conv}{conv}
\DeclareMathOperator{\ort}{O}

\DeclareMathOperator{\iso}{Iso}

\begin{document}

\title{Polarization of lattices: Stable cold spots and spherical designs}

\address{C.~Bachoc, Universit\'e de Bordeaux,
Institut de Math\'ematiques de Bordeaux,
351, cours de la Lib\'eration,
33405, Talence cedex,
France}

\address{P.~Moustrou, Institut de Mathe\'ematiques de Toulouse, UMR
  5219, UT2J, 31058 Toulouse, France}

\address{F.~Vallentin, Department Mathematik/Informatik, Abteilung
  Mathematik, Universit\"at zu K\"oln, Weyertal~86--90, 50931 K\"oln,
  Germany}

\address{M.C.~Zimmermann, Department Mathematik/Informatik, Abteilung
  Mathematik, Universit\"at zu K\"oln, Weyertal~86--90, 50931 K\"oln,
  Germany}
  
\author{Christine Bachoc}
\email{christine.bachoc@u-bordeaux.fr}
  
\author{Philippe Moustrou}
\email{philippe.moustrou@math.univ-toulouse.fr}

\author{Frank Vallentin}
\email{frank.vallentin@uni-koeln.de}

\author{Marc Christian Zimmermann}
\email{marc.christian.zimmermann@gmail.com}

\date{November 23, 2025}

\maketitle

\markboth{C. Bachoc, P. Moustrou, F. Vallentin, and M.C. Zimmermann}{Polarization of lattices: Stable cold spots and spherical designs}

\begin{abstract}
We consider the problem of finding the minimum of inhomogeneous Gaussian lattice sums: Given a lattice $L \subseteq \R^n$ and a positive constant $\alpha$, the goal is to find the minimizers of $\sum_{x \in L} e^{-\alpha \|x - z\|^2}$ over all $z \in \R^n$.

By a result of B\'etermin and Petrache from 2017 it is known that for steep potential energy functions---when $\alpha$ tends to infinity---the minimizers in the limit are found at deep holes of the lattice. In this paper, we consider minimizers which already stabilize for all $\alpha \geq \alpha_0$ for some finite $\alpha_0$; we call these minimizers stable cold spots.

Generic lattices do not have stable cold spots. For several important lattices, like the root lattices, the Coxeter-Todd lattice, and the Barnes-Wall lattice, we show how to apply the linear programming bound for spherical designs to prove that the deep holes are stable cold spots. We also show, somewhat unexpectedly, that the Leech lattice does not have stable cold spots. 
\end{abstract}

\setcounter{tocdepth}{1} 

\tableofcontents

\section{Introduction}

Let $L \subseteq \mathbb{R}^n$ be an $n$-dimensional lattice in
Euclidean space, and let $f \colon [0,\infty) \to \mathbb{R}$ be a
nonnegative function. The \textit{polarization} of the lattice $L$ with respect to the potential function $f$ is defined by
\[
\mathcal{P}(f, L) = \min_{z \in \mathbb{R}^n} p(f,L,z) \; \text{ with } \; p(f,L,z) = \sum_{x \in L} f(\|x - z\|^2) .
\]
One physical interpretation, proposed by Borodachov, Hardin, and Saff \cite[Chapter 14]{Borodachov-Hardin-Saff-2019}, is as follows:
If $f(\|x - z\|^2)$ represents the amount of a substance received at a point $z$ due to an injector located at $x$, which points receive the least substance when injectors are placed at all lattice points $x \in L$?

\subsection{Max-min polarization for compact manifolds}

Another closely related and natural problem is the \textit{max-min polarization problem} for compact manifolds. Given a compact manifold $\mathcal{M}$ with metric $d$, a potential function $f$, and a natural number $N$, the problem is to distribute $N$ points $x_1, \ldots, x_N$ on $\mathcal{M}$ so as to maximize the polarization of these points with respect to $f$. What is the optimal value of the bilevel optimization problem
\[
\max_{x_1, \ldots, x_N \in \mathcal{M}} \; \min_{z \in \mathcal{M}} \; \sum_{i=1}^N f(d(x_i, z)^2) \;\; ?
\]
As the sphere covering problem is the inhomogeneous counterpart of the sphere packing problem, the max-min polarization problem can similarly be viewed as the inhomogeneous variant of the \textit{potential energy minimization problem}
\[
\min_{x_1, \ldots, x_N \in \mathcal{M}} \sum_{i,j = 1, i \neq j}^N f(d(x_i,x_j)^2),
\]
which has been extensively studied, especially over the last two decades, after Cohn and Kumar \cite{Cohn2007a} introduced the notion of universally optimal point configurations. Due to its bilevel structure, the max-min polarization problem is significantly more difficult than the potential energy minimization problem. Only very few exact results are known.

For example, Borodachov \cite{Borodachov2022b} investigates the max-min polarization problem on the unit sphere $\mathcal{M} = S^{n-1}$ with $N = n+1$ points. He proves that the vertices of a regular simplex provide the unique optimal solution for a wide range of potential functions. While the symmetry of the regular simplex makes the result intuitive, Borodachov shows that even for well-behaved potential functions (decreasing and convex), two distinct cases have to be analyzed:
\begin{enumerate}
\item[(a)] When the minimizers $z$ lie at the vertices of the polar simplex (this happens when the derivative $f'$ is concave),
\item[(b)] When the minimizers $z$ coincide with the vertices of the regular simplex itself (this occurs if the derivative $f'$ is convex). 
\end{enumerate}
This fundamental distinction was first observed by Stolarsky \cite{Stolarsky-1975}.

\subsection{Max-min polarization for lattices}

When the manifold $\mathcal{M}$ is the non-compact Euclidean space $\R^n$, then the max-min polarization problem, restricted to lattices, becomes
\[
  \max_{L \subseteq \mathbb{R}^n} \min_{z \in \R^n} \sum_{x \in L} f(d(x, z)^2) = \max_{L \subseteq \mathbb{R}^n} \mathcal{P}(f,L),
\]
where one maximizes over $n$-dimensional lattices $L$ of a given point density.

For the choice of potential function, the \textit{Gaussian core model} is often considered; In this case, \textit{inhomogeneous Gaussian lattice sums} with
\[
f_\alpha(r) = e^{-\alpha r} \; \text{ and } \; p(f_\alpha, L, z) = \sum_{x \in L} e^{-\alpha \|x - z\|^2}
\]
are studied. This choice of potential functions is natural due to a theorem by Bernstein, see \cite[Theorem 9.16]{Simon-2011}, which states that every completely monotonic function $f : (0, \infty) \to \R$, that is, $f$ is $C^\infty$ and $(-1)^k f^{(k)} \geq 0$ for all $k \geq 0$, can be expressed as
\[
f(r) = \int_0^\infty e^{-\alpha r} \, d\mu(\alpha),
\]
for some measure $\mu$ on $[0,\infty)$. In particular, when a lattice is optimal for the max-min polarization problem for every $f_{\alpha}$, with $\alpha > 0$, then the lattice is called \textit{universally optimal for polarization}.

Recently, the max-min polarization problem for two-dimensional lattices was solved. B\'etermin, Faulhuber, and Steinerberger \cite{Betermin-Faulhuber-Steinerberger-2021} proved that the hexagonal lattice---the $A_2$ root lattice---is universally optimal for polarization among all two-dimensional lattices with the same point density.

\subsection{Deep holes and cold spots}

To solve this bilevel optimization, a thorough understanding of the inner minimization problem is required.
In the two-dimensional case, Baernstein II \cite{Baernstein-1997} analyzed the situation when $L = A_2$ is the hexagonal lattice. For every positive $\alpha$, he showed that $z \mapsto p(f_\alpha, L, z)$ attains its minimum at a deep hole $c$ of the lattice. A \textit{deep hole} of a lattice is a point $c$ that maximizes the distance to the nearest lattice points: The point $c$ satisfies
\[
\|x - c\| = \max_{z \in \R^n} \min_{x \in L} \|x - z\|.
\]
In his proof, Baernstein II used the maximum principle for the heat equation. He considers $p(f_\alpha,L,z)$ as the solution of the heat equation for the point $z$, where unit heat sources are located at the lattice points. Then for fixed $\alpha$, which represents the reciprocal of time, minimizing $p(f_\alpha,L,z)$ amounts to finding the temperature 
at the coolest points in the plane. Inspired by this physical interpretation, we propose to call the minimizers of $z \mapsto p(f_\alpha,L,z)$ \textit{cold spots}. We say that a cold spot is \textit{universal}, if it is a cold spot for every $\alpha > 0$, and thus for every completely monotonic potential. 

Now Baernstein's result states that for the hexagonal lattice, the deep holes are universal cold spots. However, he also shows that for two-dimensional lattices, other than the hexagonal lattice, deep holes and cold spots generally differ.

Finding the cold spots is, in general, a difficult problem. Currently no computational method to exactly determine cold spots exist. While one can compute partial sums of the series
$p(f_\alpha, L, z)$ and approximate local minima numerically, obtaining rigorous performance guarantees would require significant additional work.

In the limiting case of very steep potential functions, as $\alpha \to \infty$, one naturally expects that the max-min polarization converges to the lattice sphere covering problem
\[
\min_{L \subseteq \R^n} \max_{z \in \R^n} \min_{x \in L} \|x - z\|,
\]
where the minimization is over $n$-dimensional lattices of given point density.

This is indeed the case, as shown by B\'etermin and Petrache \cite[Theorem 1.5]{Betermin-Petrache-2017}: They proved that as $\alpha \to \infty$ the cold spots of $L$ tend to deep holes. More precisely, they showed that there exists a deep hole $c$ of $L$ such that, for any $z \in \R^n$, there exists a threshold $\alpha_z$ such that for all $\alpha > \alpha_z$,
\[
\sum_{x \in L} e^{-\alpha\|x-z\|^2} \geq \sum_{x \in L} e^{-\alpha\|x-c\|^2}.
\]

\subsection{First observations about lattices with ``highly symmetric" deep holes} \label{sec:Introduction:highlysymmetric}

The question now arises: Under which conditions do cold spots \textit{coincide} with deep holes, not just in the limit, but already for finite $\alpha$. This is known to occur for the hexagonal lattice. We call a cold spot \textit{stable}, if it is a cold spot for every $\alpha \geq \alpha_0$ for some finite $\alpha_0$.  One expects that stable cold spots are deep holes for lattices whose deep holes are ``highly symmetric". 

We start by some first, easy observations. If the potential function $f$ is differentiable, then the critical points of the function $z \mapsto p(f, L, z)$ are given by the points $z$ for which the gradient
\[
\nabla p(f, L, z) = \sum_{x \in L} 2 f'(\|x-z\|^2) (x-z)
\]
vanishes. The Hessian is
\[
\nabla^2 p(f, L, z) = \sum_{x \in L} \left(2 f'(\|x-z\|^2) I_n + 4
f''(\|x-z\|^2) (x-z)(x-z)^{\sf T}\right),
\]
where $I_n$ denotes the identity matrix with $n$ rows/columns.

When the lattice is highly symmetric around $z$, one can simplify the computation of $p(f, L, z)$ as well as its gradient and Hessian. For this we group the summands occurring in the above series into \textit{inhomogeneous shells} around the point $z$
\[
L(z,r) = \{x \in L : \|x-z\| = r\}.
\]
We are interested in the case that all these nonempty, inhomogeneous shells carry spherical designs, where we say that a finite set $X$ on the sphere of radius $r$ with center $z$, denoted by
\[
S^{n-1}(z,r) = \{x \in \R^n : \|x - z\| = r\}
\]
is a \textit{spherical $M$-design} if every polynomial
$p$ of total degree at most $M$ has the same average over the sphere as over
the set $X$. That is,
\[
\int_{S^{n-1}(0,r)} p(y) \, d\omega(y) = \frac{1}{|X|} \sum_{x \in X} p(x-z)
\]
holds for every such polynomial $p$. Here we integrate with respect to the rotationally invariant probability measure on $S^{n-1}(0,r)$. The maximum number $M$ for which $X$ is a spherical $M$-design is called the \textit{strength} of the spherical design.

The fact that $X$ is a spherical $1$-designs amounts to the balancedness of $X$ around~$z$,
\[
\sum_{x \in X} (x - z) = 0,
\]
in other words, the point $z$ is the barycenter of $X$.
Clearly, when every non-empty inhomogeneous shell $L(z,r)$, with $r \geq 0$, forms a spherical $1$-design, then $z$ is a critical
point of $z \mapsto p(f, L, z)$. 
Interestingly, for the family $f_\alpha$ of Gaussian potential functions also the following converse holds.

\begin{theorem}
\label{thm:stable-critical-points}
Let $L \subseteq \mathbb{R}^n$ be an $n$-dimensional lattice and let $f_\alpha(r) = e^{-\alpha r}$ be a Gaussian potential function. 
For a point $z$ the following statements are equivalent:
\begin{enumerate}
  \item \label{thm:stable-critical-points:1}Every non-empty inhomogeneous shell $L(z,r)$, with $r \geq 0$, forms a spherical $1$-design.
  \item \label{thm:stable-critical-points:2}$z$ is critical for $z \mapsto p(f_\alpha,L,z)$ for all $\alpha > 0$.
  \item \label{thm:stable-critical-points:3}There exists an $\alpha_0$ such that $z$ is critical for $z \mapsto p(f_\alpha,L,z)$ for every $\alpha \geq \alpha_0$.
\end{enumerate}
We will call a point $z$ satisfying the equivalent conditions above a \emph{stable} critical point.
\end{theorem}

\begin{proof}
Clearly, \eqref{thm:stable-critical-points:1} $\Rightarrow$ \eqref{thm:stable-critical-points:2} $\Rightarrow$ \eqref{thm:stable-critical-points:3}.
We are left to show that \eqref{thm:stable-critical-points:3} implies \eqref{thm:stable-critical-points:1}.

So assume \eqref{thm:stable-critical-points:3} and write
\[
\nabla p(f_\alpha, L, z) = \sum_{x \in L} -2\alpha e^{-\alpha \|x-z\|^2} (x-z) = \sum_{r \geq 0} -2\alpha e^{-\alpha r^2} \sum_{x \in L(z,r)} (x-z).
\]
Now assume that there exists some $r$ such that $\sum_{x \in L(z,r)} (x-z) \neq 0$, that is the shell $L(z,r)$ is not a spherical $1$-design.
We write $d_r = \sum_{x \in L(z,r)} (x-z)$ and choose $r_0$ such that $L(z,r_0)$ is non-empty, $d_{r_0} \neq 0$ and $r_0$ is as small as possible.
Then 
\[
  \nabla p(f_\alpha, L, z) = 0 \quad \Longleftrightarrow \quad d_{r_0} + \sum_{r > r_0} e^{-\alpha (r^2-r_0^2)} d_r = 0.
\]
Now $d_{r_0} \neq 0$ and all $d_r$ are independent of $\alpha$ but 
\[
  \sum_{r > r_0} e^{-\alpha (r^2-r_0^2)} d_r = -d_{r_0} \neq 0
\]
is not.
If we bound this term, coordinate-wise, by absolute value, we obtain a strictly monotonically decreasing function in $\alpha$, which tends to $0$.
But then it cannot be true that $d_{r_0} + \sum_{r > r_0} e^{-\alpha (r^2-r_0^2)} d_r = 0$ for all $\alpha \geq \alpha_0$, contradicting assumption~\eqref{thm:stable-critical-points:3}.
\end{proof}

We want to extract some geometric intuition from this theorem. For this, it is convenient to briefly discuss the Voronoi cells and the Delaunay polytopes of a lattice. For more on Delaunay polytopes, spherical designs and the lattice sphere covering problem we refer to \cite{DSV-2008} and \cite{DSV-2012}.

The Voronoi cell
\[
V(L) = \{y \in \mathbb{R}^n : \|y - x\| \leq \|y\| \text{ for all } x \in L\}
\]
of a lattice $L$ tiles space through lattice translations and provides  a polytopal decomposition of $\R^n$. 
The vertices of this polytopal decomposition are lattice translates of the vertices of $V(L)$ and are called holes of the lattice. Those with maximum distance to the nearest lattice vectors are called deep holes of the lattice and this distance is the \textit{covering radius} $\mu(L)$ of $L$.

The polytopal decomposition that is geometrically dual to the Voronoi decomposition is called the Delaunay subdivision. This subdivision is composed of Delaunay polytopes, where each point $z \in \R^n$ determines a Delaunay polytope by forming the convex hull of all lattice points closest to $z$. Specifically, the Delaunay polytope is the convex hull of the non-empty inhomogeneous shell $L(z,r)$ with the smallest possible radius $r \geq 0$. Holes are exactly the centers of the circumsphere of full-dimensional Delaunay polytopes.

Theorem~\ref{thm:stable-critical-points} implies that stable critical points are quite special. Candidates for stable critical points can be identified using the Delaunay decomposition. For a point to be a stable critical point, it is necessary to be the barycenter of the vertices of its Delaunay polytope, thus the barycenter of the vertices of a face of a full-dimensional Delaunay polytope. 
This condition does not hold for generic points, whose Delaunay polytope is $0$-dimensional and consists only of the nearest lattice point, see Figure~\ref{fig:DelaunayA2}.

 \begin{figure}[htb]
  \centering
    \begin{subfigure}[b]{0.4\textwidth}
        \centering
        \includegraphics[]{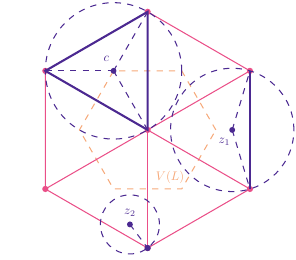}
        \caption*{The Delaunay polytope is full-dimensional only at the holes. In general, the Delaunay polytope is a face of a full-dimensional Delaunay polytope.}
    \end{subfigure}%
    \hspace{0.1\textwidth}
    \begin{subfigure}[b]{0.4\textwidth}
        \centering
        \includegraphics[]{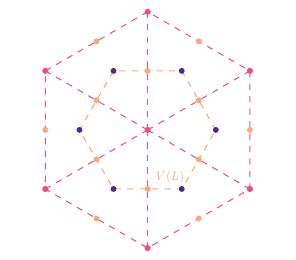}
        \caption*{Thus the only candidates for stable critical points are barycenters of the vertices of these faces, when they coincide with the center of their circumsphere.}
    \end{subfigure}
     \caption{The Delaunay decomposition of a lattice and the candidates for stable cold spots.}
     \label{fig:DelaunayA2}
 \end{figure}

This strong condition further implies that lattices possessing stable cold spots are also quite special. Since stable cold spots can only occur at deep holes, a necessary condition for a deep hole to be a stable cold spot is that it is the barycenter of the vertices of its Delaunay polytope. This condition is not satisfied for generic lattices, see Figure~\ref{fig:DelaunayGeneric}. 

 \begin{figure}[htb]
  \centering
    \begin{subfigure}[b]{0.4\textwidth}
        \centering
        \includegraphics[]{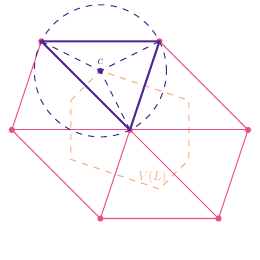}
        \caption*{For generic lattices, the deep holes are not the barycenter of their Delaunay polytopes.}
    \end{subfigure}%
    \hspace{0.1\textwidth}
    \begin{subfigure}[b]{0.4\textwidth}
        \centering
        \includegraphics[]{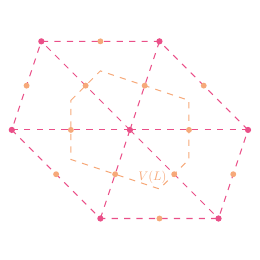}
        \caption*{Generically, the only candidates for stable critical points are lattice points and half-lattice points.}
    \end{subfigure}
     \caption{Deep holes are not stable cold spots for generic lattices.}
     \label{fig:DelaunayGeneric}
 \end{figure}

One reason why every inhomogeneous shell $L(z,r)$ could be balanced comes from central symmetry. We have $L(z,r) - z = -(L(z,r) - z)$ if and only if $2z$ is a lattice vector. This applies, for example, to lattice points and to midpoints of the facets of 
the Voronoi cell $V(L)$ because of a famous characterization of facet midpoints due to Voronoi \cite{Voronoi-1908}: a point $y$ is a facet midpoint of $V(L)$ if and only if $2y \in L$ and $\pm 2y$ are the unique shortest vectors in the coset $2y + 2L$.

\medskip

If $X$ is a spherical $2$-design, we have
\[
\sum_{x \in X}  (x-z) (x-z)^{\sf T} = \frac{r^2 |X|}{n} I_n,
\]
and therefore, when every inhomogeneous shell is a spherical $2$-design, the Hessian simplifies to
\begin{equation}
\label{eq:Hessian-2-design}
\nabla^2 p(f, L, z) = \sum_{r \geq 0} |L(z,r)| \left(2f'(r^2) + \frac{4r^2}{n} f''(r^2) \right)I_n.
\end{equation}

\subsection{Summary of the main results}

The main result of this paper, proved in Section~\ref{sec:proof-main-theorem}, is that in lattices
with ``highly symmetric" deep holes, the deep holes are equal to cold spots for $f_{\alpha}$ when $\alpha$ is large enough. 

\begin{theorem}
\label{thm:main}
Let $L \subseteq \R^n$ be an $n$-dimensional lattice and let $c \in \mathbb{R}^n$ be a deep hole of $L$. Suppose that, up to equivalence (by the affine isometry group of the lattice), $c$ is the only deep hole of $L$. Suppose further that every inhomogeneous shell around $c$ is a spherical $2$-design. Then $c$ is a stable cold spot. 
\end{theorem}

In our proof we combine two different bounds for inhomogeneous Gaussian lattice sums by a covering argument. 
The first bound is from B\'etermin, Petrache \cite{Betermin-Petrache-2017}, which is helpful to bound $p(f_{\alpha},L,z)$ for points $z$ far away from deep holes.
The second bound is new and comes from an application of linear programming bounds for spherical designs. With this bound one can show that deep holes are local minimizers for $z \mapsto p(f_{\alpha},L,z)$. Furthermore, it also yields very explicit, quantitative information about the inhomogeneous Gaussian lattice sum in the neighborhood of $c$, which is needed for the covering argument.

\smallskip

In Section~\ref{sec:inhomogeneous-shells} we show that our main theorem is applicable to all root lattices.

\begin{theorem} \label{thm:rootlattices}
The deep holes of root lattices are stable cold spots.
\end{theorem}

For this we determine the stabilizer subgroup of the affine isometry group of the lattice at a deep hole and see that all inhomogeneous shells around deep holes carry spherical $2$-designs. In the situation of root lattices we can improve the generic covering argument used to prove Theorem~\ref{thm:main}. By this we get reasonable estimates for the threshold $\alpha_0$. For example, we shall show that for $L = E_8$ already $\alpha_0 = 23$ suffices (Theorem~\ref{thm:e8}) and for $L = D_4$ already $\alpha_0 = 5$ suffices (Theorem~\ref{thm:d4}).

\smallskip

Using the same strategy we shall see in Section~\ref{sec:other-lattices} that the dual of the exceptional root lattices $E_6^*$, $E_7^*$, as well as the Coxeter-Todd lattice in dimension $12$ and the Barnes-Wall lattice in dimension $16$ can be treated by Theorem~\ref{thm:main}. In all these cases the deep holes are stable cold spots.

\smallskip

However, the behaviour of the Leech lattice is somewhat unexpected.

\begin{theorem}
\label{thm:leech-lattice}
The Leech lattice does not have stable cold spots.
\end{theorem}

To prove this, we consider the $23$ inequivalent deep holes of the Leech lattice in Section~\ref{sec:leech-lattice}, determine the one to which cold spots converge when $\alpha$ tends to infinity, and then apply Theorem~\ref{thm:stable-critical-points}.

\subsection{Open questions}

It would be interesting to prove or disprove in all the treated cases, whether the stable cold spots are universal cold spots. Currently, we do not know how to attack this question. It is clear that the techniques developed in this paper will not suffice, and new ideas are needed. Also we cannot resolve the cases $A_n^*$, for $n \geq 3$, and $D_n^*$, for $n \geq 5$. In these cases inhomogenous shells around deep holes are all spherical $1$-designs but not spherical $2$-designs.

\section{Stable cold spots and spherical $2$-designs: Proof of Theorem~\ref{thm:main}}
\label{sec:proof-main-theorem}

Before going into details, we first set up notations and sketch the strategy of our proof. 
Let $L$ be an $n$-dimensional lattice. 
Up to translation by lattice vectors, it is enough to consider the restriction of the function $z \mapsto p(f_\alpha, L, z)$ to the Voronoi cell $V(L)$.
Consequentely, in the rest of this section, unless explicitely stated, we assume that $z$ and $c$ belong to $V(L)$, and that $c$ is a deep hole of $L$. The assumptions of Theorem~\ref{thm:main} ensure that all the deep holes have the same potential. 

Our proof then goes in two steps. 
First, Theorem~\ref{lem:locmin} shows that if every shell around $c$ is a spherical $2$-design, then $c$ is a local minimizer of $z \mapsto p(f_\alpha, L, z)$, and moreover we find explicit $\alpha_c$ and $R_{\alpha_c}$ such that $c$ is the unique minimizer in the ball $B(c,R_{\alpha_c})$, for every $\alpha > \alpha_c$.
This will be proven in Section~\ref{sec:LP}, and is our main contribution in this section.
Then, with Theorem~\ref{lem:locmin}, it is enough to study the potential in a ball centered at $0$ and of radius $\varrho < \mu(L)$, the covering radius of $L$.
This is the purpose of Theorem~\ref{cor:rho}, and is a consequence of previous work by B\'etermin and Petrache, that we discuss in Section~\ref{sec:BP}.

\subsection{Bounding inhomogeneous Gaussian lattice sums: Far away from deep holes}\label{sec:BP}

We revisit some of the arguments originally due to B\'etermin and Petrache \cite[Theorem 1.5]{Betermin-Petrache-2017}, in order to obtain the following result that meets our needs:

\begin{theorem}
\label{cor:rho}
  Let $c\in V(L)$ be a deep hole of $L$, and let $\varrho$ such that $\varrho < \mu(L)$. 
  Then there exists $\alpha_\varrho$ such that for any $\alpha > \alpha_\varrho$, for every $z$ in the ball $B(0, \varrho)$, 
  \[
  p(f_\alpha, L, z) > p(f_\alpha, L, c). 
  \]
  \end{theorem}

\smallskip

This result is an immediate consequence of the following inequality:

\begin{lemma}\label{thm:ineqBP}
Let $L$ be an $n$-dimensional lattice, let $c$ be a deep hole of $L$, and let $z \in \mathbb{R}^n$ be an arbitary point. If $\alpha > n / (2 \mu(L)^2)$, then
\begin{equation}
\label{eq:Betermin-Petrache}
\left(\sum_{x \in L} e^{-\alpha \|x-c\|^2} \right) \bigg/
\left(\sum_{x \in L} e^{-\alpha \|x-z\|^2} \right)
\leq
e^{-\alpha(\mu(L)^2 - \|z\|^2)} \left(\frac{2\alpha \mu(L)^2 e}{n} \right)^{n/2}.
\end{equation}
\end{lemma}

Indeed, for every $z$ in $B(0, \varrho)$,
\[
\mu(L)^2 - \|z\|^2 > \mu(L)^2 - \varrho^2 > 0,
\]
 and 
  \[
    e^{-\alpha(\mu(L)^2 - \varrho^2)} \left(\frac{2\alpha\mu(L)^2 e}{n} \right)^{n/2}
  \]
  goes to $0$ when $\alpha$ tends to infinity.  
Note that this proof gives an explicit estimate on $\alpha_\varrho$ depending on $\varrho$. 

It remains to prove Lemma~\ref{thm:ineqBP}.
For this, we need estimates by Banaszczyk \cite{Banaszczyk-1993}, which in the following variant can be found in \cite[Lemma 18.2 and Problem 18.4.1]{Barvinok}.

\begin{lemma}\label{lem:Bana}
Let $L$ be an $n$-dimensional lattice. Then:
\begin{enumerate}
\item[(a)] For every $r > \sqrt{\frac{n}{2\pi}}$ and every $z\in \R^n$,
\[ 
\sum_{x \in L, \| x - z \| > r } e^{-\pi \| x - z \|^2} \leq e^{-\pi r^2}\left(\frac{2\pi e r^2}{n}\right)^{n / 2} 
\sum_{x \in L} e^{-\pi \|x\|^2}.
\]
\item[(b)] For every $\alpha>0$ and every $z\in \R^n$,
\[
\sum_{x \in L} e^{-\alpha \|x\|^2} \leq e^{\alpha \|z\|^2} 
\sum_{x \in L} e^{-\alpha \|x-z\|^2}.
  \]
\end{enumerate}
\end{lemma}

\begin{proof}[Proof of Lemma~\ref{thm:ineqBP}]
We can write
\[
\sum_{x \in L} e^{-\alpha\| x - c\|^2} = \sum_{x \in \sqrt{\alpha/\pi} L} e^{-\pi \| x - \sqrt{\alpha/\pi} c\|^2}. 
\]
Now $\sqrt{\alpha/\pi} c$ is a deep hole of the scaled lattice $\sqrt{\alpha/\pi} L$. For a given radius $r$, we can split the last sum into two parts,
\[
\sum_{x \in \sqrt{\alpha/\pi } L, \| x - \sqrt{\alpha/\pi} c\| \leq r } e^{-\pi \| x - \sqrt{\alpha/\pi} c\|^2} + \sum_{x \in \sqrt{\alpha/\pi} L, \| x - \sqrt{\alpha/\pi} c\| > r } e^{-\pi \| x - \sqrt{\alpha/\pi} c\|^2}. 
\]
Observe that the first sum vanishes whenever $r$ is strictly smaller than $\sqrt{\alpha/\pi} \mu(L)$, the covering radius of $\sqrt{\alpha/\pi} L$. If $\alpha$ is large enough, we can choose $r$ in the range
$\sqrt{\frac{n}{2\pi}} < r < \sqrt{\alpha/\pi} \mu(L)$ so that we can apply (a) of Lemma~\ref{lem:Bana}. This gives
\[
\begin{split}
  \sum_{x \in L} e^{-\alpha \|x-c\|^2} 
  & = \sum_{x \in \sqrt{\alpha/\pi} L, \| x - \sqrt{\alpha/\pi} c\| > r } e^{-\pi \| x - \sqrt{\alpha/\pi} c\|^2}\\
  & \leq e^{-\pi r^2} 
  \left(\frac{2 \pi e r^2}{n} \right)^{n/2}
  \sum_{x \in \sqrt{\alpha/\pi} L} e^{-\pi\|x\|^2}.
\end{split}
\]
Now we apply (b) of Lemma~\ref{lem:Bana},
\[
\sum_{x \in L} e^{-\alpha \|x-c\|^2}
\leq e^{-\pi r^2} 
\left(\frac{2 \pi e r^2}{n} \right)^{n/2}
e^{-\alpha\|z\|^2} \sum_{x \in L} e^{-\alpha\|x-z\|^2}.
\]
By letting $r$ tending to $\sqrt{\alpha/\pi} \mu(L)$ we get the desired inequality \eqref{eq:Betermin-Petrache}.
\end{proof}

Based on Lemma~\ref{lem:Bana}, B\'etermin and Petrache obtain the following statement \cite[Theorem 1.5]{Betermin-Petrache-2017}, which is an immediate consequence of Theorem~\ref{cor:rho}.

\begin{corollary}
\label{cor:Betermin-Petrache}
  Let $L$ be an $n$-dimensional lattice and let $c$ be a deep hole of
  $L$. Then for any $z\in \R^n$, which is not a deep
  hole\footnote{This assumption is missing in the original statement \cite[Theorem 1.5]{Betermin-Petrache-2017}}, there exists $\alpha_z$ such that for any $\alpha > \alpha_z$,
\[
p(f_\alpha, L, z) > p(f_\alpha, L, c). 
\]
Thus, for $\alpha \to \infty$ cold spots tend to deep holes.
\end{corollary}

Note that the arguments given above are not strong enough to prove that cold spots are deep holes for finite $\alpha$.  The quantity $\mu(L)^2 - \| z \|^2$ cannot be uniformly bounded from below with respect to $z$ in the Voronoi cell of $L$.

\subsection{Bounding inhomogeneous Gaussian lattice sums: Using the
  linear programming bound for spherical designs around deep holes}
  \label{sec:LP}

Here we prove that if every shell around a deep hole is a spherical $2$-design, then the deep hole is a (strict) local minimizer when $\alpha$ is large enough. 
That is, we show:

\begin{theorem}\label{lem:locmin}
Let $L$ be an $n$-dimensional lattice and $c$ be a deep hole of $L$. 
Assume that all inhomogeneous shells around $c$ are spherical $2$-designs.
Then there exist $\alpha_c$ and $R_{\alpha_c}$ such that for every $\alpha \geq \alpha_c$ and for every $z\neq c$ in the ball $B(c,R_{\alpha_c})$, strict inequality $p(f_\alpha, L, z) > p(f_\alpha, L, c)$ holds.
\end{theorem}
Moreover, explicit estimates on $\alpha_c$ and $R_{\alpha_c}$ will be given in Lemma~\ref{lem:Ralpha}.

\medskip

Qualitatively, it is easy to see that $c$ is a local minimizer for $\alpha$ large enough, using the computation of the Hessian in~\eqref{eq:Hessian-2-design}, which for the Gaussian core model evaluates to
\[
\nabla^2 p(f_\alpha, L, c) = \sum_{r \geq 0} |L(c,r)| \left( \frac{4\alpha^2 r^2}{n} - 2\alpha\right) e^{-\alpha r^2} I_n.
\]
Indeed, if $\alpha$ is sufficiently large, then all summands are positive multiples of the identity matrix.
This ensures the existence of a radius $R_\alpha>0$ such that $c$ is a minimizer in the ball $B(c,R_\alpha)$.

However, as we aim to combine this local information with Theorem~\ref{cor:rho} to obtain a global bound, we need to derive a uniform lower bound for $R_\alpha$ when $\alpha$ is large enough. 
The formula for the Hessian is easy to evaluate at a point $z$ only when all inhomogeneous shells around $z$ are spherical $2$-designs. So information about the Hessian at points close to $c$ seem to be difficult to get. A more delicate analysis is required.

For this, we will heavily use the linear programming bound for spherical designs, orginally developed by Fazekas and Levenshtein \cite{Fazekas1995a} and Yudin
  \cite{Yudin1995a} to determine lower bounds for the covering radius of spherical
designs. Later, in the context of polarization problems on spheres, this linear programming bound was modified in \cite{Boyvalenkov2022a} and \cite{Borodachov2022a}. 

\medskip

Let us recall the details of this linear programming bound. 
Let $a : [-1,1] \to \R$ be any function. We are
  interested in determining
  \[
  E_a(X) = \min\left\{\sum_{x \in X} a(x \cdot y) : y \in S^{n-1}\right\}.
  \]
  
  To state the linear programming bound for $E_a$ we need to introduce
  specific polynomials $P^n_k$.
  They are univariate polynomials of degree $k$,
  normalized by $P^n_k(1) = 1$, and orthogonal in the following sense
  \[
    \int_{-1}^1 P^n_k(t) P^n_l(t) (1-t^2)^{(n-3)/2}
    dt = 0 \quad \text{ if } k \neq l.
  \]
  In fact, the polynomials $P^n_k$ are appropriate multiples of the
  Jacobi polynomials $P^{(n-3)/2,(n-3)/2}_k$, see for instance the book
  \cite{Andrews1999a} by Andrews, Askey, and Roy. Note that we have
  \[
  P_0^n(t) = 1, \quad P_1^n(t) = t,\quad P_2^n(t) = \frac{1}{n-1}(nt^2-1).
  \]
  
  \begin{lemma}
   \label{lem:lp-method}
    Let $X \subseteq S^{n-1}$ be a spherical $M$-design and let $a :
    [-1, 1] \to \R$ be any function. Let $h(t) = \sum_{k = 0}^M h_k
    P^n_k(t)$ be a polynomial so that $h(t) \leq a(t)$ for all $t \in
    [-1,1]$, then the inequality $h_0 |X| \leq  E_a(X)$ holds.
  \end{lemma}
  
  \begin{proof}
    The inequality follows easily from orthogonality and the spherical
    design property. For $y \in S^{n-1}$,
  \[
  \begin{split}
      E_a(X) & \geq \sum_{x \in X} a(x \cdot y) \geq \sum_{x \in X} h(x \cdot
        y) = |X| \frac{1}{|X|} \sum_{x \in X} \sum_{k=0}^M h_k P_k^n(x \cdot y)\\
        & = |X| \sum_{k= 0}^M h_k \int_{S^{n-1}} P_k(x \cdot y) \, d\omega(x)\\
        & = |X| \sum_{k= 0}^M h_k \int_{-1}^1 P_k(t) (1-t^2)^{(n-3)/2} \, dt\\
        & = |X| h_0. \qedhere
  \end{split}
  \]
  \end{proof}
  
  As usual, the proof of the previous lemma provides additional
  information about sharp cases; that is, when there is a polynomial $h$
  and a point $y \in S^{n-1}$ with
  \[
    h_0 |X| = E_a(X) = \sum_{x \in X} a(x \cdot y).
  \]
  Indeed, one sees from the proof that this can only happen when
  $a(x \cdot y) = h(x \cdot y)$ for $x \in X$. Since we also need $h(t) \leq a(t)$ for
  all $t \in [-1,1]$ this also means that, if $a$ is differentiable, $h$
  and $a$ need to coincide up to second order when
  $x \cdot y \in (-1,1)$ for $x \in X$.
  
 \medskip 

To bound $p(f_\alpha, L, z)$ for points $z$ close to a deep hole $c$ we apply the linear programming bound for every concentric sphere $z = c + \rho y$, with $\rho > 0$ and $y \in S^{n-1}$, and for every inhomogeneous shell $L(c,r)$;
\[
\begin{split}
p(L,f_\alpha, z) & = \sum_{x \in L} e^{-\alpha \|x - z\|^2}
= \sum_{x \in L} e^{-\alpha\|x - c - \rho y\|^2} \\
& = \sum_{r \geq 0} e^{-\alpha r^2} e^{-\alpha \rho^2}
\sum_{x \in L(c,r)} e^{2\alpha \rho r ((x-c)/r) \cdot  y};
\end{split}
\]
see Figure~\ref{fig:strategy-LP}.

\begin{figure}[htb]
  \begin{center}
   \includegraphics[width=8cm]{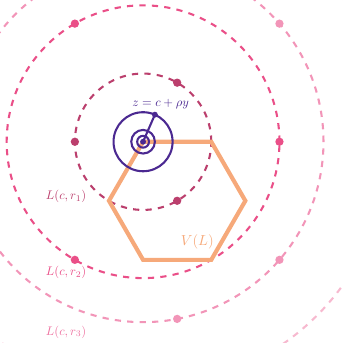}
  \end{center}
    \caption{Our strategy is to bound the inhomogeneous Gaussian lattice sum close to
      the deep hole $c$ using the linear programming bound for \textit{every}
      inhomogeneous shell $L(c,r)$ and for the points on \textit{every}
      concentric sphere $c + \rho y$, with $y \in S^{n-1}$, around the
      deep hole $c$.}
    \label{fig:strategy-LP}
\end{figure}

\begin{lemma}\label{lem:LPsol}
Suppose the inhomogeneous shell $L(c,r)$ is a spherical $2$-design. Then 
\begin{equation}
  \label{eq:bound}
\sum_{x \in L(c,r)} e^{2\alpha \rho r ((x-c)/r) \cdot  y} \geq |L(c,r)| b(\alpha, \rho, r)
\end{equation}
 holds with
\[
b(\alpha, \rho, r) = 
\frac{ne^{\frac{2\alpha \rho r}{n}}+e^{-2\alpha \rho r}}{n+1},
\]
and if $L(c,r)$ is additionally centrally symmetric, then
\[
b(\alpha, \rho, r) = \cosh\left(\frac{2\alpha \rho r}{\sqrt{n}}\right).
\]
\end{lemma}

\begin{proof}
\begin{enumerate}
\item
We apply Lemma \ref{lem:lp-method} to  $X = \{ \frac{x-c}{r} :  x\in L(c,r)\}$, and functions $a(t)$ of the form 
  \[
  a_\gamma(t) = \exp(\gamma t) \quad \text{with $\gamma = 2\alpha \rho r$.}
  \]
Then a feasible solution for Lemma~\ref{lem:lp-method} is given by:
\[
h_\gamma(t) = h_0P_0^n(t) + h_1P_1^n(t) + h_2P_2^n(t)
\]
with
\begin{align*}
h_0 &= \frac{ne^{\frac{\gamma}{n}}+e^{-\gamma}}{n+1}
\\
h_1 &= \frac{(\gamma n^2 - \gamma +2n)e^{\frac{\gamma}{n}}-2ne^{-\gamma}}{(n+1)^2}
\\
h_2 &= \frac{((\gamma-1)n^2 - \gamma + n )e^{\frac{\gamma}{n}}+ n(n-1) e^{-\gamma}}{(n+1)^2}.
\end{align*}
The polynomial $h$ interpolates the function $a_\gamma(t)$ at the points $t = -1$ with multiplicity $1$, and $t = 1/n$ with multiplicity $2$. By applying Lemma~2.1 from \cite{Cohn2007a}, we have $h_\gamma(t) \leq \exp(\gamma t)$ on $[-1,1]$.

\item
Because the inhomogeneous shell $L(c,r)$ is centrally symmetric we can rewrite the sum
\[
\sum_{x \in L(c,r)} e^{2\alpha \rho r ((x-c)/r) \cdot  y}  = 
\sum_{x \in L(c,r)} \cosh(2 \alpha \rho r ((x-c)/r) \cdot y).
\]
We apply Lemma~\ref{lem:lp-method} to functions $a(t)$ of the form 
\[
a_\gamma(t) = \cosh(\gamma t) \quad \text{with $\gamma = 2\alpha \rho r$.}
\]
Then a feasible solution for Lemma~\ref{lem:lp-method} is given by:
\[
h_\gamma(t) = 
 \cosh\left(\frac{\gamma}{\sqrt{n}}\right) P_0^n(t) + \frac{\gamma}{2}\sinh\left(\frac{\gamma}{\sqrt{n}}\right)\frac{n-1}{\sqrt{n}}P_2^n(t).
\]
Now that $a_\gamma$ is even, we may assume that $h_\gamma$ is even. 
This leaves freedom only to the coefficients $h_0$ and $h_2$. We use Hermite interpolation at $t=\frac{1}{\sqrt{n}}$ at order $2$, and apply again \cite[Lemma~2.1]{Cohn2007a} to show $h_\gamma(t) \leq \exp(\gamma t)$ on $[-1,1]$. 
\end{enumerate}
\end{proof}

\begin{remark}
  \label{rem:optimality}
  The bounds of the previous lemma are in fact optimal.
  
  In the first case, the bound is tight, when $X$ are the vertices of a regular simplex and when $y = - v$, where $v$ is any vertex of the simplex. 
  Then the inner products between $y$ and the elements of $X$ are the interpolation points $-1$ and $1/n$. 
  Therefore,
  \[
  E_{a_\gamma}(X) \leq \sum_{x\in X} a_\gamma (x \cdot y) = \sum_{x\in X} h_\gamma (x \cdot y) = h_0 |X| \leq  E_{a_\gamma}(X),
  \]
  which shows the optimality of the bound.

In the second case, the bound is tight when $X$ are the vertices of a cross polytope, namely  $X=\{\pm e_i\}$, where $e_i$ is the $i$-th vector of the canonical basis. For $y = \frac{1}{\sqrt{n}} \sum_{i = 1}^n e_i$, then the inner products between $y$ and the elements of $X$ are the interpolation points $\pm \frac{1}{\sqrt{n}}$, which proves optimality.
\end{remark}

With Lemma~\ref{lem:LPsol}, we have
\[
p(L,f_\alpha, z) \geq \sum_{r \geq 0} |L(c,r)| e^{-\alpha r^2} e^{-\alpha \rho^2}
 b(\alpha, \rho, r).
\]
  
On the other hand,
\[
p(f_\alpha, L, c)  = \sum_{r \geq 0} |L(c,r)| e^{-\alpha r^2} .
\]
Thus, if for a given $\alpha$, we find $R$ such that for every $0< \rho < R$, and for every $r$ the strict inequality
\begin{equation}
  \label{eq:ineq}
e^{-\alpha \rho^2} b(\alpha,\rho,r) > 1
\end{equation}
holds, then we get that for every $z\neq c$ in the ball $B(c,R)$, 
\[
p(f_\alpha, L, z) - p(f_\alpha, L, c) \geq \sum_{r \geq 0}  |L(c,r)| e^{-\alpha  r^2}  (e^{-\alpha \rho^2} b(\alpha,\rho,r) - 1) > 0.
\]

We need to verify inequality \eqref{eq:ineq}. Note that since both expressions $b(\alpha, \rho, r)$ are increasing with $r$, if \eqref{eq:ineq} is satisfied for $r = \mu(L)$, then it is satisfied for every $r \geq \mu(L)$. 

\smallskip

We therefore consider the function
\[
g(\rho) = e^{-\alpha \rho^2}b(\alpha,\rho,\mu(L)) - 1,
\]
and look for $R_\alpha>0$ such that for every $0 < \rho < R_\alpha$, $g(\rho)>0$.

\begin{lemma}
\label{lem:Ralpha}
In the general case, when the inhomogeneous shells are spherical 2-designs,
then for every $\alpha$ such that $\alpha > \frac{n^2}{ \mu(L)^2}(\frac{n+1}{2n} + \log(\frac{n+1}{n})^2)$, then $g(\rho) > 0$ on the interval $\left(0, R_\alpha \right]$, where
\[
R_\alpha = \frac{\mu(L)}{n} + \sqrt{\frac{\mu(L)^2}{n^2}- \frac{\log(\frac{n+1}{n})}{\alpha}}.
\]
When, moreover, the inhomogeneous shells are centrally symmetric, then the inequality 
 $\alpha > \frac{n}{ \mu(L)^2}(1/2 + \log(2)^2)$ is sufficient, and we can take 
 \[
R_\alpha = \frac{\mu(L)}{\sqrt{n}} + \sqrt{\frac{\mu(L)^2}{n}- \frac{\log(2)}{\alpha}}.
\]
\end{lemma}

\begin{proof}
We focus on the general case where 
\[
g(\rho) = e^{-\alpha \rho^2}\left(\frac{ne^{\frac{2\alpha \rho \mu(L)}{n}}+e^{-2\alpha \rho \mu(L)}}{n+1}\right) - 1,
\]
because the second, more specific case, is similar.
The strategy consists in using two different estimates: a first one to ensure the positivity of $g(\rho)$ for $\rho$ in an interval of the form $[r_\alpha, R_\alpha]$, and second one, using the Taylor expansion of $g$ around $0$, to make sure that $g(\rho) >0$ for $\rho \in (0,r_\alpha)$.  

First, we use the trivial lower bound 
\[
\frac{ne^{\frac{2\alpha \rho \mu(L)}{n}}+e^{-2\alpha \rho \mu(L)}}{n+1} > \frac{ne^{\frac{2\alpha \mu(L) \rho}{{n}}}}{n+1},
\]
which gives
\[
g(\rho) > \frac{ne^{\alpha \rho\left(\frac{2\mu(L)}{n}-\rho\right)}-(n+1)}{n+1}
\]
and therefore $g(\rho) >0$ whenever
\[
\alpha \rho\left(\frac{2\mu(L)}{n}-\rho\right) > \log\left(\frac{n+1}{n}\right).
\]
Because $\frac{n+1}{2n} + \log\left(\frac{n+1}{n}\right)^2 > \log\left(\frac{n+1}{n}\right)$, the assumption on $\alpha$ ensures that $g(\rho)>0$ on the interval 
\[
\left[ \frac{\mu(L)}{{n}} - \sqrt{\frac{\mu(L)^2}{n^2} - \frac{\log\left(\frac{n+1}{n}\right)}{\alpha}},  \frac{\mu(L)}{{n}} + \sqrt{\frac{\mu(L)^2}{n^2} - \frac{\log\left(\frac{n+1}{n}\right)}{\alpha}} \right]
.\]

It remains to deal with small $\rho$. Note that because $\sqrt{1-x} \geq 1-x$ for every $0\leq x \leq 1$, we have
\[
\frac{\mu(L)}{n} - \sqrt{\frac{\mu(L)^2}{n^2} - \frac{\log\left(\frac{n+1}{n}\right)}{\alpha}} = \frac{\mu(L)}{n} - \frac{\mu(L)}{n}\sqrt{1 - \frac{\log\left(\frac{n+1}{n}\right)n^2}{\alpha\mu(L)^2}} \leq \frac{\log\left(\frac{n+1}{n}\right){n}}{\alpha\mu(L)}
\]
and therefore it is sufficient to prove that $g(\rho)>0$ on $\left(0,\frac{\log\left(\frac{n+1}{n}\right){n}}{\alpha\mu(L)} \right]$.
For $\rho$ in this interval, we have
\[
e^{-\alpha \rho^2} \geq 1 -\alpha \rho^2 > 0.
\]
Moreover, since
\[
e^{\frac{2\alpha \rho \mu(L)}{n}} > 1 + \frac{2\alpha \rho \mu(L)}{n} + \frac{2\alpha^2 \rho^2 \mu(L)^2}{n^2}
\]
and
\[
e^{-2\alpha \rho \mu(L)} > 1 - 2\alpha \rho \mu(L),
\]
then 
\[
ne^{\frac{2\alpha \rho \mu(L)}{n}}+e^{-2\alpha \rho \mu(L)}> n+1 + \frac{2\alpha^2 \rho^2 \mu(L)^2}{n}.
\]
We get
\[
g(\rho) > (1 -\alpha \rho^2)\left(1 + \frac{2\alpha^2 \mu(L)^2 \rho^2}{n(n+1)}\right) - 1 = -\alpha \rho^2\left( 1 - \frac{2\alpha \mu(L)^2}{n(n+1)} + \frac{2\alpha^2 \mu(L)^2 \rho^2}{n(n+1)}  \right),
\]
which is positive between $0$ and
\[
\sqrt{\frac{1}{\alpha} - \frac{n(n+1)}{2\alpha^2 \mu(L)^2}} = \sqrt{\frac{\alpha - \frac{n(n+1)}{2\mu(L)^2}}{\alpha^2}}. 
\]
Thus $g(\rho) > 0$ on $\left(0,\frac{\log\left(\frac{n+1}{n}\right){n}}{\alpha\mu(L)} \right]$ whenever $\alpha - \frac{n(n+1)}{2\mu(L)^2} > \frac{\log\left(\frac{n+1}{n}\right)^2n^2}{\mu(L)^2}$, which is ensured by the assumption on $\alpha$.

For the second case, we use the similar estimates
\[
\cosh\left(\frac{2\alpha \mu(L) \rho}{\sqrt{n}}\right) > \frac{e^{\frac{2\alpha \mu(L) r}{\sqrt{n}}}}{2}, 
\]
and
\[
g(\rho) > (1 -\alpha \rho^2)\left(1 + \frac{2\alpha^2 \mu(L)^2  \rho^2}{n}\right) - 1 = -\alpha \rho^2\left( 1 - \frac{2\alpha \mu(L)^2}{n} + \frac{2\alpha^2 \mu(L)^2 \rho^2}{n}  \right).
\]
\end{proof}

From Lemma~\ref{lem:Ralpha}, we can easily deduce Theorem~\ref{lem:locmin}:

\begin{proof}[Proof of Theorem~\ref{lem:locmin}]
In the general case (respectively the centrally symmetric case), we can take any $\alpha_c> \frac{n^2}{ \mu(L)^2}(\frac{n+1}{2n} + \log(\frac{n+1}{n})^2)$ (respectively $\alpha_c > \frac{n}{ \mu(L)^2}(1/2 + \log(2)^2)$).
Then, because the function $\alpha \mapsto R_\alpha$ is increasing with $\alpha$, $c$ is the unique minimizer of $z\mapsto p(f_\alpha,L,z)$ in the ball $B(c,R_{\alpha_c})$, for every $\alpha \geq \alpha_c$.
\end{proof}

\subsection{Putting the two bounds together, finishing the proof of Theorem~\ref{thm:main}}\label{sec:mainproof}

The assumptions of Theorem~\ref{thm:main} allow us to apply Theorem~\ref{lem:locmin}.
We have for every $\alpha \geq \alpha_c$ and for every $z \neq c$ in the ball $B(c,R_{\alpha_c})$, 
\[
p(f_\alpha, L, z) > p(f_\alpha, L, c).
\]
Since the deep holes of $L$ are precisely the points in $V(L)$ having norm $\mu(L)$, there exists $0 < \varrho < \mu(L)$ such that 
\[
V(L) \setminus \bigcup_{c} B(c, R_{\alpha_c}) \subseteq B(0, \varrho),
\]
where the union is taken over all the deep holes in $V(L)$.
Figure~\ref{fig:strategy-gen} illustrates this covering argument.

\begin{figure}[hbt]
  \begin{center}
   \includegraphics[width=7cm]{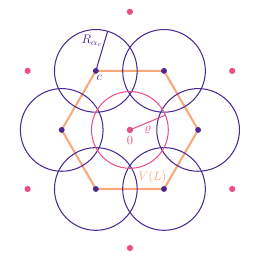}
  \end{center}
    \caption{Covering $V(L)$ by balls of two kinds.}
    \label{fig:strategy-gen}
\end{figure}

We can now apply Theorem~\ref{cor:rho}, which provides  $\alpha_\varrho$ such that for every $\alpha > \alpha_\varrho$ and every $z \in B(0, \varrho)$, 
\[
p(f_\alpha, L, z) > p(f_\alpha, L, c).
\]
The desired result then follows by taking $\alpha_0 = \max\{\alpha_c, \alpha_\varrho\}$.

\qed

\section{Cold spots of root lattices}
\label{sec:inhomogeneous-shells}

We now move to prepare the proof of Theorem \ref{thm:rootlattices}, asserting that for large enough $\alpha$ the cold spots of a root lattice are equal to its deep holes.

For this we firstly discuss how inhomogeneous energy behaves if a lattice is decomposable, i.e. an orthogonal sum of sublattices. 
This allows us to restrict all investigations to indecomposable root lattices, about which we collect some basic facts in Section \ref{sec:special:points:lattices} and Appendix \ref{sec:Appendix:root-lattices}.

In lieu of Theorem \ref{thm:main} we then only need to show that for an indecomposable root lattice $L$ and a deep hole $c$ of $L$ the inhomogeneous shells 
\[
  L(c,r) = \{x \in L : \|x-c\| = r\}
\]
form spherical $2$-designs whenever they are non-empty.
This will be achieved in Theorem \ref{thm:deep-hole-designs} and is based on the fact that the stabilizer group of $c$ contains a Weyl group which induces the necessary design strength.

While, for the present application, it is sufficient that the stabilizer group of the deep hole $c$ contains a suitable Weyl group, we can give a more precise characterization of the stabilizer group, which we present in Lemma \ref{lem:deephole:stabilizer}.

We conclude the section by illustrating how to compute an explicit $\alpha_0$ for the application of Theorem \ref{thm:main}. 
We do this for the root lattices $D_4$ and $E_8$ in Theorem \ref{thm:d4} and Theorem \ref{thm:e8}.

\subsection{Decomposable lattices}

We will call a lattice $L$ in $\R^n$ \emph{decomposable} if we can find non trivial sublattices $L_1$ on $V_1 = \R L_1$ and $L_2$ on $V_2 = \R L_2$ with
\[
  L = L_1 \perp L_2,
\]
otherwise $L$ is called \emph{decomposable}.
A general feature of decomposable lattices is that (inhomogeneous) energy factors multiplicatively through the orthogonal summands.

\begin{lemma} \label{lem:decomposable}
Let $L\subset\R^n$ be a lattice that can be decomposed as the orthogonal direct sum of lattices
\[
L = L_1 \perp L_2 \perp \ldots \perp L_m,
\]
namely there are subspaces $V_1, \ldots, V_m \subset \R^n$ such that $L_i$ is a lattice in $V_i$ for every $1 \leq i \leq m$, and $\R^n = V_1 \perp V_2 \perp \ldots \perp V_m$.
Then for any $\alpha > 0$ and $z = z_1 + \ldots + z_m$ with $z_i \in V_i$, we have
\[
p(f_\alpha, L, z) = p(f_\alpha, L_1, z_1)p(f_\alpha, L_2, z_2)  \cdots p(f_\alpha, L_m, z_m).
\]
In particular, if for every $1 \leq i \leq m$, $c_i$ is a mimimizer of $z_i \mapsto p(f_\alpha, L_i, z_i)$, then $c = c_1 + \ldots + c_m$ is a minimizer of $z \mapsto p(f_\alpha, L, z)$.
\end{lemma}

\begin{proof}
It is enough to prove the statement when $L = L_1 \perp L_2$. 
Then, for $z = z_1 + z_2$, we have
\begin{align*}
p(f_\alpha, L, z) & = \sum_{\ell \in L}e^{-\pi\alpha \| \ell - z\|^2}
\\ & = \sum_{\ell_1 \in L_1, \ell_2 \in L_2}e^{-\pi\alpha \| \ell_1 + \ell_2 - z_1 - z_2\|^2}
\\ & = \sum_{\ell_1 \in L_1, \ell_2 \in L_2}e^{-\pi\alpha \| \ell_1  - z_1 \|^2}e^{-\pi\alpha \|  \ell_2 - z_2\|^2}
\\ & = \left(\sum_{\ell_1 \in L_1}e^{-\pi\alpha \| \ell_1  - z_1 \|^2}\right)\left(\sum_{\ell_2 \in L_2}e^{-\pi\alpha \|  \ell_2 - z_2\|^2}\right)
\\ & = p(f_\alpha, L_1, z_1)p(f_\alpha, L_2, z_2).
\end{align*}
\end{proof}

\subsection{Designs on inhomogeneous shells of indecomposable root latttices} \label{sec:special:points:lattices}

An integral lattice $L$ is a \emph{root lattice} if and only if 
\begin{equation} \label{eq:root:lattice}
  L = \langle L(2) \rangle_{\Z},
\end{equation}
i.e. if $L$ is generated by its elements of squared norm $2$, which in fact is a root system.
Note that as a direct consequence of the definition all root lattices are even, that is $\|v\|^2 \in 2\Z$ for all $v \in L$.

If $L$ is decomposable, say $L = L_1 \perp \ldots \perp L_m$ with $L_1,\ldots,L_m$ indecomposable, then this is true for the root system
\[
  L(2) = L_1(2) \perp \ldots \perp L_m(2).
\]
It is well known that in this case each root system $L_i(2)$ is indecomposable\footnote{Recall that a root system is called \emph{decomposable} if $R = R_1 \perp R_2$ for root systems $R_1$, $R_2$.} as a root system and is of type $A_n$ ($n \geq 2$), $D_n$ ($n \geq 3$) or $E_n$ ($n \in \lset 6,7,8 \rset$) (c.f. \cite[Theorem 1.2]{Ebeling2013}).

For a root lattice $L$ we write $W(L)$ for its \emph{Weyl group}, which is the Weyl group of the root system $R=L(2)$ (c.f. Appendix \ref{sec:Appendix:weyl}).
Of course $W(L)$ is a subgroup of $\ort(L)$.
Similarly we will write $W^a(L)$ for the affine Weyl group of $L$, which is the semi-direct product 
\[
  W^a(L) = W(L) \rtimes L \subseteq \iso(L) = \ort(L) \rtimes L,
\]
where $L$ acts on itself by translations and $\iso(L)$ is the \emph{affine isometry group} of $L$.

A very nice property of these root systems of type $A_n$, $D_n$ and $E_n$ is that they are self-dual, that is 
\[
  A_n = (A_n)^\vee, D_n = (D_n)^\vee, \text{ and }  E_n = (E_n)^\vee.
\]
So, in particular, if $R = L(2)$, then $R = R^{\vee}$ (c.f. Appendix \ref{sec:Appendix:weyl} for reference).

From the above discussion we collect information on \emph{special} points of root lattices, that is points in space which are stabilized by a suitably nice group of affine isometries of the lattice.
It turns out that these are precisely the points of the \emph{dual} of $L$, which, in general, is the lattice
\[
  L^* = \lset x \in \R^n : x \cdot v \in \Z \text{ for all } v \in L \rset.
\]

\begin{lemma} \label{lem:special:points:rootlattice}
  Let $L$ be an $n$-dimensional root lattice with root system $R = L(2)$ and Weyl group $W(L)$.
  For $z \in \R^n$ we have
  \[
    W^a_z(L) \cong W(L) \quad \Longleftrightarrow \quad z \in L^*.
  \]
\end{lemma}

\begin{proof}
  This is based on the more general concept of special points for Weyl groups, on which we collect some information in Appendix \ref{sec:Appendix:weyl} but in general we refer to \cite{Bourbaki1968}.
  
  Explicitly we note that for a point $z$ the isomorphism $W^a_z(L) \cong W(L)$ is equivalent to $z$ being a special point of $W^a(L)$ (\cite[Ch. V, \S 3, Proposition 9]{Bourbaki1968}) and the special points of root lattices are precisely the elements of the weight lattice $L^*$ (\cite[Ch.VI, \S 2, Proposition 3]{Bourbaki1968}).
\end{proof}

It is the unimodularity of $E_8$, i.e. $E_8 = E_8^*$, that keeps us from including it in the framework used in Lemma \ref{lem:special:points:rootlattice}.
In particular we already see that special points for $E_8$ which are not lattice elements will have a stabilizer that is smaller than $W(E_8)$, we will adress this case in Lemma \ref{lem:deephole:stabilizer}.

For the subsequent discussion we need some facts about these lattices, which we collect in Table \ref{tab:root-lattices}.
In particular we emphazise that for indecomposable root lattices not of type $E_8$ all deep holes are elements of the associated dual lattice and therefore special points for the associated affine Weyl group.
To justify the contents of this table we refer to Appendix \ref{sec:Appendix:root-lattices} and more generally to \cite[Chapter 4]{Conway1988a}.

\begin{center}
\begin{table}[!h]
\begin{tabular}{ c  p{9em}  p{14em}  c }
$R$  & $L^* / L$ & special points of $W^a(L)$ & $[\ort(L):W(L)]$\\ \hline 
\rule{0pt}{4mm}%
$A_n$ & cyclic of order $n+1$ & all vertices of a fundamental simplex for $W^a(A_n)$, which are all lattice points and holes. & $2$ \\[6ex]
$D_n$ & $n$ even: \newline cyclic of order $4$ 
\newline  $n$ odd: \newline non-cyclic of order $4$ & all lattice points and holes & $2$ \\[9ex]
$E_6$ & cyclic of order $3$ & all lattice points and holes & $2$ \\[1ex]
$E_7$ & cyclic of order $2$ & all lattice points and holes & $1$\\[1ex]
$E_8$ & trivial & all lattice points & $1$\\
\hline
\end{tabular}
\caption{Data regarding indecomposable root lattices.}
\label{tab:root-lattices}
\end{table}
\end{center}

With this we are ready to discuss the main result of this section.

\begin{theorem}\label{thm:deep-hole-designs}
  Let $L\subseteq \R^n$ be a root lattice and $c \in \R^n$ a deep hole of $L$. Then for $r > 0$ the inhomogeneous shell $L(z,r)$ is either empty or a spherical $2$-design.
\end{theorem}

\begin{proof}
  We can assume that $L$ is an indecomposable root lattice, as every decomposable root lattice splits orthogonally into an orthogonal sum of indecomposable root lattices. 
    
  The main tool we use is Theorem 3.12 in \cite{Goethals1981}, which states that for any finite subgroup $G$ of the orthogonal group $\ort(\R^n)$ the following conditions are equivalent
  \begin{enumerate}
    \item every $G$-orbit is a spherical $t$-design;
    \item \label{enum:noharm} there are no $G$-invariant harmonic polynomials of degrees $1,2,\ldots,t$.
  \end{enumerate} 

  For each irreducible root lattice we find a suitable group $G$ which satisfies \eqref{enum:noharm} with $t=2$ above and has the property that each shell $L(c,r)$ is a union of orbits of $G$. 
  Then each orbit is a spherical $2$-design and so this true for $L(c,r)$.

  Let $c$ be a deep hole of $L$ and $r$ such that $L(c,r)$ is non-empty.
  Let $\iso(L)_c$ be the stabilizer of $c$ in $\iso(L)$.
  By Lemma \ref{lem:special:points:rootlattice} we have that if $L \not\cong E_8$, then the stabilizer of a deep hole $c$ of $L$ in $W^a(L) \subseteq \iso(L)$ contains\footnote{In fact the stabilizer $\iso(L)_c$ can be computed explicitly, we will do so in Lemma \ref{lem:deephole:stabilizer}, however, this is not necessary for the argument here.} (an isomorphic copy $G$ of) the Weyl group $W(L)$ of $L$.

  For $L \cong E_8$ we find that there is $G \subseteq \iso(L)_c$, with $G \cong W(D_8)$, this can be checked by explicit computation.
  In fact $\iso(L)_c \cong W(D_8)$, we prove this in Lemma \ref{lem:deephole:stabilizer} \eqref{lem:deephole:stabilizer:E}.

  So in all cases $G$ can be chosen to be a Weyl group of type $A$, $D$, or $E$ of full dimension with respect to $L$.
  Note that $G$ is a group of isometries fixing $c$ and therefore also fixes $L(c,r)$ set-wise, so $L(c,r)$ is indeed a union of orbits of $G$.

  We can apply \cite[Theorem 3.12]{Goethals1981} to these groups, since in all cases such a Weyl group has no harmonic invariants of degree $1$ or $2$.
  This follows from \cite[Theorem 4.6]{Goethals1981}, the succeeding discussion and the data in Table 1 of \cite{Goethals1981}.
\end{proof}

For all root lattices which are not of type $E_8$ the proof of the above Theorem extends naturally to all special points in Table \ref{tab:root-lattices}.

\begin{corollary}
  Let $L \not \cong E_8$ be an irreducible root lattice and let $c$ be a hole of $L$. 
  Then $L(c,r)$ is either empty or a spherical $2$-design.
\end{corollary}

\subsection{The full stabilizer group for special points and indecomposable root lattices}  \label{sec:special:points:stabilizer}

The proof of Theorem \ref{thm:deep-hole-designs} relies only on the fact that for root lattices the full stabilizer subgroup of a deep hole contains a Weyl group.
We proved the existence of such a subgroup in Lemma \ref{lem:special:points:rootlattice} for all cases except $E_8$, this case we will handle now.
Along with that we compute the full stabilizer subgroups of deep holes in the affine isometry group of indecomposable root lattices.
This is achieved in Lemma \ref{lem:deephole:stabilizer}.
While this is not necessary for the proof of Theorem \ref{thm:deep-hole-designs} we are not aware of a reference to this problem in the literature, and it seems to be an obvious question how far the full stabilizer of a deep hole in a (indecomposable) root lattice is away from the subgroup identified in Lemma \ref{lem:special:points:rootlattice}.
\medskip

Let $\iso(L)$ be the group of affine isometries of $L$. 
Every element in $\iso(L)$ can be written as the composition of an element $f\in \ort(L)$ by a translation $t_u$ by a vector of $u\in L$.
Recall that for $c\in \R^n$
\[
  \iso(L)_c = \lset \phi \in \iso(L) : \phi(c) = c \rset \subseteq \iso(L)
\]
be the stabilizer of $c$ in $\iso(L)$. 

Firstly, we determine how the translational part and the linear part of an affine isometry $\phi$ are related if $\phi$ is in the stabilizer $\iso(L)_c$ of a point $c$.

\begin{lemma} \label{lem:stabilizer:affine}
Let $\phi = t_u \circ f \in \iso(L)$. Then $u \in L$ and $\phi \in \iso(L)_c$ if and only if $u = c-f(c)$.
\end{lemma}
\begin{proof}
The first assertion is clear since
\[
  u = \phi(0) \in L
\]
by $\phi \in \iso(L)$.
If $c = \phi(c)$ we find
\[
c = \phi(c) = f(c) + u \Rightarrow u = c - f(c).
\]
If on the contrary $u = c - f(c)$ then
\[
  \phi(c) = f(c) + c - f(c) = c.
\]
\end{proof}

Before determining $\iso(L)_c$, where $c$ is a deep hole of an irreducible root lattice $L$, we quickly recall the structure of the groups $W(L)$ and $\ort(L)$ for the $A_n$ and $D_n$ series specifically.
Note that $D_3 \cong A_3$, we therefore consider the $D_n$ series only for $n \geq 4$.
In the case of the $A_n$ series we have $W(A_n) \cong S_{n+1}$, where $S_m$ is the permutation group on $m$ elements, and $\ort(A_n) \cong \lset \pm \id \rset \times S_{n+1}$ for the full orthogonal group.
In the case of the $D_n$ series we have that $W(D_n)$ is generated by all permutations on the $n$ coordinates, together with an even number of sign changes.
Abstractly we can write $W(D_n) \cong \lset \pm 1 \rset^n_+ \rtimes S_n \cong \lset \pm 1 \rset^{n-1} \rtimes S_n$, where $\lset \pm 1 \rset^n_+$ operates by an even number of sign changes on the coordinates. 
The full orthogonal group is generated by the Weyl group and one additional symmetry and (abstractly) is a semi-direct or wreath product $\ort(D_n) \cong \lset \pm 1 \rset^{n} \rtimes S_{n+1} \cong \lset \pm 1 \rset \wr S_n$.
For the explicit description of $D_n$ as in Appendix \ref{sec:Appendix:root-lattices}, the additional symmetry separating $\ort(D_n)$ and $W(D_n)$ can be choosen to be the sign change of the last coordinate.
For lattices of the $E_n$-series we have that $\ort(E_7) = W(E_7)$, while $\ort(E_6) = \lset \pm \id \rset \times W(E_6)$.

We also quickly collect some information on the first shell $L(c,\mu(L))$ around a deep hole $C$ in an indecomposable root lattice $L$, in particular for lattices of type $A_n$ or $D_n$.
Recall that the points $L(c,\mu(L))$ are the vertices of the Delaunay polytope around $c$.

For $A_n$ a typical deep hole is given by the vector $c = [\lfloor (n+1)/2 \rfloor]$ (c.f. Appendix \ref{sec:Appendix:root-lattices}) and the covering radius is $\mu(A_n) = 1/2$ if $n$ is even and $\mu(A_n) = ab/(n+1)$ if $n$ is odd, where $a = \lfloor (n+1)/2\rfloor$ and $b = \lceil (n+1)/2\rceil$.
For $n=2$ we have the hexagonal lattice $A_2$ and the first shell forms the set of vertices of a regular simplex.
For $n$ odd the first shell corresponds to the vertices of a centrally symmetric polytope, while this is wrong for $n$ even. 
To be more precise the vertices of the associated Delaunay polytope are given by the $\binom{n+1}{n+1/2}$ permutations of the vector
\[
  ( \underbrace{-\tfrac{1}{2},\cdots,-\tfrac{1}{2}}_{\tfrac{n+1}{2} \text{ times}},\underbrace{\tfrac{1}{2},\cdots,\tfrac{1}{2}}_{\tfrac{n+1}{2} \text{ times}}) \in A_n
\]
if $n$ is odd, and the $\binom{n+1}{a}$ permutations of the vector
\[
  ( \underbrace{-\tfrac{b}{n+1},\cdots,-\tfrac{b}{n+1}}_{a \text{ times}},\underbrace{\tfrac{a}{n+1},\cdots,\tfrac{a}{n+1}}_{b \text{ times}}) \in A_n
\]
if $n$ is even.
We refer to \cite[Section 4]{Conway1991} for more details.

For $D_n$, we can describe the Delaunay polytopes around all holes in terms of the representatives $[1],[2],[3]$ as given in Appendix \ref{sec:Appendix:root-lattices}.
The Delaunay polytopes around $[1]$ and $[3]$ are half-cubes (also known as demi-cubes or parity-polytopes) which are centrally symmetric if and only if $n$ is even.
For $[2]$ the Delaunay polytope is a cross-polytope, which is centrally symmetric for arbitary $n$.
Now for $n=3$ there is one type of deep hole, represented by $[2]$, for which the Delaunay polytope is a cross-polytope.
For $n=4$ there are three types of deep holes, represented by $[1],[2],[3]$ and in this case the half-cube and cross-polytope coincide. 
For $n \geq 5$ there are two types of deep holes, represented by $[1]$ and $[3]$ and in these cases the Delaunay polytopes are half-cubes.

For the $E_8$ lattice the Delaunay polytope around a deep hole is a cross polytope.
The Delaunay polytope at a deep hole of $E_6$ is a polytope with $27$ vertices, known as Schl\"afli polytope $2_{21}$. 
The Delaunay polytope at a deep hole of  $E_7$ is a centrally symmetric polytope with $56$ vertices, known as Hesse polytope $3_{21}$.
A justification of the above and more details on the geometric descriptions of these polytopes can be found in \cite{Conway1991}.

\begin{lemma} \label{lem:deephole:stabilizer}
  Let $L$ be an indecomposable root lattice and let $c$ be a deep hole of $L$.
  \begin{enumerate}
    \item \label{lem:deephole:stabilizer:A} If $L$ is of type $A_n$ then,
    \[ \iso(L)_c \cong
      \begin{cases}
        W(A_n) &, \text{ if $n$ is even,} \\
        \ort(A_n) &, \text{ if $n$ is odd.}
      \end{cases}
    \]
    \item \label{lem:deephole:stabilizer:D} If $L$ is of type $D_n$ then,
    \[ \iso(L)_c \cong
      \begin{cases}
        \ort'(D_4) &, \text{ if $n = 4$,} \\
        W(D_n) &, \text{ if $n \geq 5$.}
      \end{cases}
    \]
    Here $\ort'(D_4)$ refers to a subgroup of index $4$ in $\ort(D_4)$ which contains $W(D_4)$ as a subgroup of index $2$ (c.f. the section on $D_n$ in Appendix \ref{sec:Appendix:root-lattices}).
    \item \label{lem:deephole:stabilizer:E} If $L$ is of type $E_n$ then,
    \[ \iso(L)_c \cong
      \begin{cases}
        W(E_n) &, \text{ if $n = 6,7$,} \\
        W(D_8) &, \text{ if $n = 8$.}
      \end{cases}
    \]
    Note that $\ort(E_7) \cong W(E_7)$.
  \end{enumerate}
\end{lemma}

\begin{proof}
  We prove the statements by using explicit data on the indecomposable root lattices and their Weyl and automorphism groups.
  \medskip

  We start with $L \cong E_8$\footnote{This case could be handled by a computer algebra program (such as MAGMA), but since we rely on this characterization of the stabilizer of a deep hole of $E_8$ in the proof of Theorem \ref{thm:deep-hole-designs} we give an explicit proof that does not require the use of a computer.} and use the explicit realization of $E_8$ as in Appendix \ref{sec:Appendix:root-lattices}, given by
  \begin{equation} \label{eq:E8}
    E_8 = D_8^+ = D_8 \cup \tfrac{1}{2}e + D_8.
  \end{equation}
  Let $c$ be a deep hole of $E_8$ for which $0$ is a closest lattice point. Then $\|c\|^2 = 1$ and the associated Delaunay polytope $D$ is a cross polytope.

  If $\phi \in \iso(L)_c$ then $\phi$ permutes the vertices of $D$ and therefore fixes $D$.
  We identify $\iso(L)_c$ with a subgroup $G'$ of $\ort(\R^8)$. 
  The inner automorphism $\Phi: \iso(\R^8) \rightarrow \iso(\R^8); \phi \mapsto \varphi = t_{-c} \circ \phi \circ t_{c}$ induces the isomorphism
  \[
    \iso(L)_c \cong G' = t_{-c} \circ \iso(L)_c \circ t_{c},
  \]
  where $G'$ is a subgroup of $\ort(\R^8)$ as desired. 
  In agreement with Lemma \ref{lem:stabilizer:affine} we explicity obtain
  \begin{equation} \label{eq:Gprime}
    G' = \lset \varphi \in \ort(E_8) : \varphi(c)-c \in E_8 \rset.
  \end{equation}
  We claim that $G' \cong W(D_8)$.

  For this we fix the deep hole $c = e_1$, for which the Delaunay polytope is
  \[
    D = e_1 + \conv\left(\lset \pm e_1, \ldots, \pm e_8 \rset \right) = e_1 + D'.
  \]  
  Let $\phi \in \iso(L)_{e_1}$, so it stabilizes $D$ and its center $e_1$, and so $\varphi = t_{-e_1} \circ \phi \circ t_{e_1}$ stabilizes $D' = -e_1 + D$.
  This is equivalent to
  \[
    \varphi\left(\lset \pm e_1,\ldots, \pm e_8 \rset\right) = \lset \pm e_1,\ldots, \pm e_8 \rset
  \]
  so $\varphi$ stabilizes a standard cross polytope centered at the origin.
  The group of isometries of the latter is precisely the Weyl group of $D_8$, which one can see for example by the description $W(D_8) \cong \lset \pm 1 \rset^8_+ \rtimes S_8$.

  On the contrary assume that $\varphi \in W(D_8)$. 
  Then, in \eqref{eq:E8}, $\varphi$ fixes the sublattice $D_8$ of $E_8$ and operates on $e$ by flipping an even number of signs.
  This gives $\varphi(e)-e \in \lset 0,2 \rset^8$, where an even number of $2$s occur, so $\varphi(e)-e \in 2D_8$.
  With this and the decomposition \eqref{eq:E8} we see that $\varphi$ maps $E_8$ to $E_8$, so $\varphi \in \ort(E_8)$.
  Moreover   
  \[
    \varphi(e_1) = \pm e_i \quad \Rightarrow \quad \varphi(e_1) - e_1 = \pm e_i -e_1 \in E_8,  
  \]
  so $\varphi \in G'$ as defined by \eqref{eq:Gprime}.
  This finishes the proof for $L \cong E_8$.
  \medskip

  We recall that if $c$ is a deep hole of an irreducible root lattice $L$ of type other than $E_8$, then $c \in L^*$.
  Furthermore, we recall that in these cases there is an isomorphic copy of $W(L)$ contained in the stabilizer of $c$ in $W^a(L) \subseteq \iso(L)$, i.e. $W(L) \hookrightarrow \iso(L)_c$.  
  Indeed this immediately follows from Lemma \ref{lem:special:points:rootlattice}, which asserts that the stabilizer of $c$ in $W^a$ is isomorphic to $W(L)$ as in the proof of Theorem \ref{thm:deep-hole-designs}.   

  For $L \cong A_n$ we again use the inner automorphism $\Phi: \iso(\R^n) \rightarrow \iso(\R^n); \phi \mapsto \varphi = t_{-c} \circ \phi \circ t_{c}$ to obtain an isomorphism
  \[
    \iso(L)_c \cong G' = t_{-c} \circ \iso(L)_c \circ t_{c},
  \]
  so we can work with a group $G'$ of linear isometries, with $W(L) \subseteq G'$.
  Let $D = c + D'$ be the Delaunay polytope of $c$. 
  So if $\phi$ stabilizes $D$, $\Phi(\phi)$ stabilizes $D'$ and vice versa.
  
  We recall that $\ort(A_n) = \lset \pm \id \rset \times W(A_n)$.
  If $n$ is even, then $D'$ is not centrally symmetric and so $G'$ cannot contain $-\id$, therefore $\iso(L)_c \cong G' = W(A_n)$.
  If $n$ is odd, then $D'$ is centrally symmetric and one can check that indeed $-\id \in G'$, so $\iso(L)_c \cong G' = \ort(A_n)$. 

  For $L \cong D_n$ we distinguish between $n = 4$ and $n \geq 5$.
  The case $n=4$ can be handled by direct computations using the data in Appendix \ref{sec:Appendix:root-lattices} or using a computer algebra program such as MAGMA, we do not include the details since the latter is easily done\footnote{For the claim about the structure of $\ort'(D_4)$ we refer to the discussion in Appendix \ref{sec:Appendix:root-lattices}.}.

  The case $n\geq 5$:
  Here $D_n$ has two types of deep holes which are inequivalent under translations by lattice elements of $D_n$.
  Explicitly we can write
  \[
    D_n^* = \left([0] + D_n\right) \cup \left([1] + D_n\right) \cup \left([2] + D_n\right) \cup \left([3] + D_n\right),
  \]
  where $[1],[2],[3]$ are holes of $D_n$ given by
  \begin{align*}
    [0] &= (0,0,\ldots,0) \text{ of squared norm } 0\\
    [1] &= (\tfrac{1}{2},\tfrac{1}{2},\ldots,\tfrac{1}{2}) \text{ of squared norm } n/4\\
    [2] &= (0,0,\ldots,1) \text{ of squared norm } 1\\
    [3] &= (\tfrac{1}{2},\tfrac{1}{2},\ldots,-\tfrac{1}{2}) \text{ of squared norm } n/4.
  \end{align*}
  Note that for $n \geq 5$ the two types of deep holes are $[1]$ and $[3]$.
  Now $\ort(D_n)$ is generated by $W(D_n)$ and an isometry interchanging $[1]$ and $[3]$, but this additional isometry cannot be contained in $G'$ and so $G' = W(D_n)$.
  So $\iso(L)_c \cong G' = W(D_n)$.

  For $E_7$ we note that $\ort(L) = W(L)$ and $E_6$ can be handled by explicit computations or by using a computer algebra program, such as Magma, we omit the details.
\end{proof}

It might be of interest that a partial reverse holds. 
If $z$ is a point, for which the stabilizer $\iso(L)_z$ of $z$ in $\iso(L)$ contains the Weyl group of the root system of $L$, then $z$ is either a hole of $L$ (if $L \not \cong E_8$) or an element of $L$.

\begin{corollary}
  For an irreducible root lattice $L \not \cong E_8$, we have
    \[
      W(L) \hookrightarrow \iso(L)_z \quad \Longleftrightarrow \quad \text{$z$ is a hole or an element of $L$.}
    \]
\end{corollary}

\begin{proof}
  The argument is as in the proof of Lemma \ref{lem:deephole:stabilizer}, we get the embedding of $W(L)$ into the stabilizer by Lemma \ref{lem:special:points:rootlattice} for all points of the dual lattice $L^*$.
  We then use the coset-decomposition of $L^*$ as in Appendix \ref{sec:Appendix:root-lattices}, from which we infer that every element of $L^*$ is either a lattice point or a hole of $L$.
  
  Explicitly for $A_n$ we have the decomposition
  \[
    A_n^* = \left([0] + A_n\right) \cup \left([1] + A_n\right) \cup \cdots \cup \left([n] + A_n\right),
  \]
  where $[0],[1],\ldots,[n]$ are the vertices of a fundamental simplex for $W^a(A_n)$.
  The elements $[1],\ldots,[n]$ are precisely the holes of $A_n$ (up to translations by lattice elements).
  
  For $D_n$ we have the decomposition
  \[
    D_n^* = \left([0] + D_n\right) \cup \left([1] + D_n\right) \cup \left([2] + D_n\right) \cup \left([3] + D_n\right),
  \]
  where $[1],[2],[3]$ are the holes of $D_n$. 

  For $E_6$ and $E_7$ it follows similarly, compare the data in Appendix \ref{sec:Appendix:root-lattices}.

\end{proof}

\subsection{Finding explicit $\alpha$} 
In order to find an explicit $\alpha_0$, we need to make explicit the covering strategy of the proof of Theorem~\ref{thm:main} described in Section~\ref{sec:mainproof}.

In the case of root lattices, the Voronoi cell is the union of the images of the fundamental simplex under the action of the Weyl group. 
In the two examples we consider next, we cover this fundamental simplex with exactly two balls.
Here is the strategy we follow in order to show that a given $\alpha_0$ works. 
First, $\alpha_0$ has to fulfill the condition of Lemma~\ref{lem:Ralpha}, and this lemma gives a radius $R_{\alpha_0}$ such that for every $\alpha \geq \alpha_0$, the deep hole $c$ is the unique minimizer for the potential in the ball $B(c,R_{\alpha_0})$. 
Then, the radius $\varrho$ of the second ball $B(0,\varrho)$ is the largest norm among the points in the fundamental simplex that are not covered by $B(c,R_{\alpha_0})$.
Denote by $S$ the intersection of the edges of the simplex with the sphere of radius $R_{\alpha_0}$ centered at $c$, and decompose the set of vertices of the fundamental simplex as $V\cup V'$, where the vertices in $V$ belong to $B(c,R_{\alpha_0})$, and those in $V'$ do not.
Then $B(c,R_{\alpha_0})$ covers the convex hull of $S\cup V$, so it is enough that $B(0,\varrho)$ covers the convex hull of $S\cup V'$, see Figure~\ref{Fig-Cover} for two examples in dimension~$2$.
It remains to check that $\alpha_0 \geq \alpha_\varrho$, where $\alpha_\varrho$ is given by Theorem~\ref{cor:rho}, which, according to Lemma~\ref{thm:ineqBP}, amounts to verifying the inequality
\[
e^{-\alpha_0(\mu(L)^2 - \varrho^2)} \left(\frac{2\alpha_0 \mu(L)^2 e}{n}\right)^{n/2} < 1.
\]

\begin{figure}
\includegraphics[scale=.5]{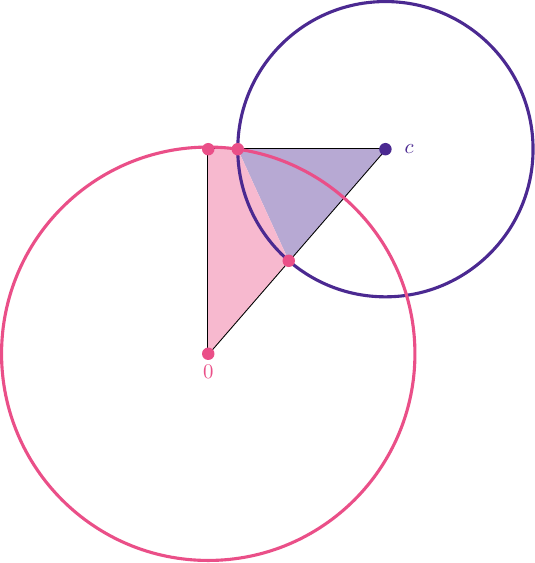} \quad \includegraphics[scale=.5]{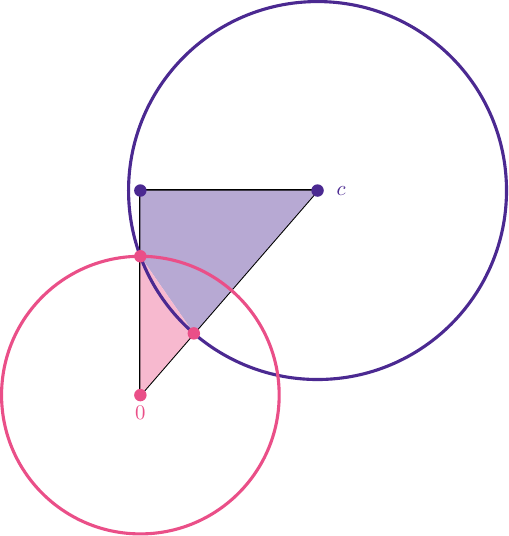} 
\caption{Covering a simplex with two spheres in dimension~2.}\label{Fig-Cover}
\end{figure}

Using this strategy, we can prove:

\begin{theorem}
\label{thm:e8}
Let $L= E_8$. 
Then for every $\alpha \geq 23$, the deep holes are the unique minimizers of $p(f_\alpha,L)$. 
\end{theorem}
\begin{proof}
First note that $\tau_L = 1$, and that all the shells around $c$ are centrally symmetric.
One can then check that $\alpha_0 = 23$ satisfies the conditions of Lemma~\ref{lem:Ralpha}, which gives $R_{\alpha_0}>0.66$.
Then, the vertices of the fundamental simplex are $0$ and

\[\begin{array}{lcl}
v_1 =(-\frac{1}{8},\frac{1}{8},\frac{1}{8},\frac{1}{8},\frac{1}{8},\frac{1}{8},\frac{1}{8},\frac{7}{8}) & \quad & v_2 = (0,0,\frac{1}{6},\frac{1}{6},\frac{1}{6},\frac{1}{6},\frac{1}{6},\frac{5}{6})
\\
v_3 = (\frac{1}{6},\frac{1}{6},\frac{1}{6},\frac{1}{6},\frac{1}{6},\frac{1}{6},\frac{1}{6},\frac{5}{6})  & \quad & 
v_4 = (0,0,0,\frac{1}{5},\frac{1}{5},\frac{1}{5},\frac{1}{5},\frac{4}{5})
\\
v_5 = (0,0,0,0,\frac{1}{4},\frac{1}{4},\frac{1}{4},\frac{3}{4}) & \quad & 
v_6 = (0,0,0,0,0,\frac{1}{3},\frac{1}{3},\frac{2}{3}) 
\\
v_7 = (0,0,0,0,0,0,\frac{1}{2},\frac{1}{2}) & \quad & 
v_8 = (0,0,0,0,0,0,0,1)
\end{array}
\]
where $v_8=c$ is a deep hole of $L$. Except $v_7$ and $0$, all the vertices  are covered by $B(c,0.66)$. 
Therefore, the sphere centered at $c$ and of radius $0.66$ intersects  the edges $[0,v_i]$ with $i\neq 7$, and the edges $[v_7,v_i]$, $i \neq 7$. 
Computing these intersections, we find that the point with the largest norm lies on the edge $[v_6,v_7]$, and has norm smaller than $\varrho = 0.72$. 
Since we also have $\| v_7 \| < \varrho$, the fundamental simplex is covered by the union of $B(c,0.66)$ and $B(0,\varrho)$. 
To conclude, we check that for $\alpha_0 = 23$,
\[
e^{- \alpha_0 (1 - \varrho^2)}  \left(\frac{ e \alpha_0}{4}\right)^4 < 1.
\]
\end{proof}

In a similar way, we get

\begin{theorem}
\label{thm:d4}
Let $L= D_4$. 
Then for every $\alpha \geq 5$, the deep holes are the unique minimizers of $p(f_\alpha,L)$. 
\end{theorem}
\begin{proof}
We follow the same strategy, with the only difference that, due to the additional symmetry of $D_4$, the fundamental simplex under the action of $W(D_4)$ is the union of three similar simplices that are equivalent with respect to the action by $O(D_4)$.
For $R_{\alpha_0} = 0.8$, all the vertices of this simplex except the origin are covered, and by looking at the edges between these vertices and $0$, we find that $\varrho = 0.35$ is sufficient, and allows to apply Theorem~\ref{cor:rho}.
\end{proof}

\section{The lattices $E_6^*$, $E_7^*$, $K_{12}$, and $BW_{16}$}
\label{sec:other-lattices}

For the dual lattices of the exceptional root lattices $E_6^*$, $E_7^*$, the Coxeter-Todd lattice $K_{12}$, and the Barnes-Wall lattice $BW_{16}$ we computed the stabilizer $\iso(L)_c$ in the affine isometry group of the lattice of a deep hole $c$. Then we determined the harmonic Molien series to see which $\iso(L)_c$-invariant harmonic polynomials do not exist, see Table~\ref{table:data-other-lattices}. 
These computations we performed with the help the computer algebra system MAGMA \cite{MAGMA}.  From this we can conclude that in all the four cases the deep holes are stable cold spots.

\begin{table}[htb]
\begin{tabular}{llll}
\hline
lattice & $|\iso(L)_c|$ & harmonic Molien series & strength \\
\hline
\rule{0pt}{4mm}%
$E_6^*$ & $1296$ & $1 + u^3 + u^4  + \cdots$ & $2$\\
$E_7^*$ & $80640$ & $1 + u^4 + 2u^6 + \cdots$ & $3$ \\
$K_{12}$ & $311040$ &  $1 + u^4 + u^5 + \cdots$ & $3$ \\
$BW_{16}$ & $743178240$ &  $1 + u^6 + 2u^8 + \cdots$ & $5$ \\
\hline
\end{tabular}
\medskip
\caption{Data for the lattices $E_6^*$, $E_7^*$, $K_{12}$, and $BW_{16}$. In all cases the inhomogeneous shells around deep holes are designs with strength at least $2$. Therefore, in all four cases stable cold spots are deep holes.}
\label{table:data-other-lattices}
\end{table}

\section{The Leech lattice}
\label{sec:leech-lattice}

The results of this section summarize parts of the master thesis of Ottaviano Marzorati (written under the supervision of the third- and fourth-named author of this paper).

\medskip

The following technical lemma gives the order of inhomogeneous Gaussian lattices sums when $\alpha$ tends to infinity. It will be crucial when analyzing the behavior of $z \mapsto p(f_{\alpha}, \Lambda, z)$ at inequivalent deep holes of the Leech lattice $\Lambda_{24}$. It also can be used to give a proof of Theorem~\ref{thm:stable-critical-points} and also to give an alternative proof of Corollary~\ref{cor:Betermin-Petrache}.

\begin{lemma}
\label{lem:estimate}
  Let $L$ be an $n$-dimensional lattice and let $z \in \mathbb{R}^n$ be any point. Then,
  \[
  \sum_{x \in L, \|x-z\| \geq r_0} e^{-\alpha \|x-z\|^2} =  \Theta(e^{-\alpha r_0^2} \alpha^{-(n-1)/2}).
  \]
  \end{lemma}
  
  \begin{proof}
  For $R_0 \geq 0$ consider the partial sum
  \[
  \sum_{r_0 \leq r \leq R_0} |L(z,r)| e^{-\alpha r^2}
  \]
  of the inhomogeneous Gaussian lattice sum $\sum_{x \in L} e^{-\alpha \|x-z\|^2}$. Using the Riemann-Stieltjes integral one can show (see for instance \cite[Appendix A]{Montgomory-Vaughan-2012}) that this partial sum equals
  \[
  A_{R_0} e^{-\alpha R_0^2} - \lim_{\varepsilon \to 0+} \int_{r_0 - \varepsilon}^{R_0} A_r (-2\alpha r) e^{-\alpha r^2 } \, dr,
  \]
  where
  \[
  A_R = \sum_{r_0 \leq r \leq R} |L(z,r)|
  \]
  is the cumulative step function of $|L(z,r)|$. When $r_0$ is fixed, it is known that $A_R = \Theta(R^n)$\footnote{Actually much more precise information about the growth of $A_R$ is known. This question is closely related to the Gauss circle problem. However, for our purpose, the given rough estimate suffices.}.
  
  Letting $R_0$ tend to infinity, this gives
  \[
  \sum_{x \in L, \|x-z\| \geq r_0} e^{-\alpha\|x-z\|^2} = \alpha \int_{r_0}^\infty \Theta(r^{n+1}) e^{-\alpha r^2} \, dr.
  \]
  Now we estimate the integral $\int_{r_0}^\infty r^{n+1} e^{-\alpha r^2} \, dr$ for $\alpha \to \infty$ using Laplace's method. Substituting $r^2 = s$ yields
  \[
  \int_{r_0}^\infty r^{n+1} e^{-\alpha r^2} \, dr = \frac{1}{2}\int_{r_0^2}^\infty s^{(n-1)/2} e^{-\alpha s} \, ds.
  \]
  Furthermore, using the Gamma function, we get
  \[
  \begin{split}
  \frac{1}{2}\int_{r_0^2}^\infty s^{(n-1)/2} e^{-\alpha s} \, ds 
  & = \frac{1}{2} e^{-\alpha r_0^2} \int_0^\infty (s+r_0^2)^{(n-1)/2} e^{-\alpha s} \, ds \\
  & = e^{-\alpha r_0^2} \int_0^\infty \Theta(s^{(n-1)/2}) e^{-\alpha s} \, ds \\
  & = e^{-\alpha r_0^2} \, \Theta(\alpha^{-(n-3)/2}). \qedhere
  \end{split}
  \]
  \end{proof}

Let $L$ be an $n$-dimensional lattice and let $z \in \R^n$ be any point. Define
  \[
  0 \leq r_1 < r_2 < \ldots
\] 
so that $L(z,r_n)$, with $n = 1, 2, \ldots$, are all inhomogenous lattice shells around $z$. Define 
\[
a_1 = |L(z,r_1)|, \; a_2 = |L(z,r_2)|, \ldots.
\]
This gives rise to an order on the points $z \in \R^n$: We set $z > z'$ if and only if
\[
(r_1,a_1, r_2, a_2, \ldots) <_{\text{lex}} (r'_1,a'_1, r'_2, a'_2, \ldots)
\]
in the following lexicographic order
\[
(r_1 < r'_1) \text{ or } (r_1 = r'_1 \text{ and } a_1 > a'_1) 
\text{ or } (r_1 = r'_1 \text{ and } a_1 = a'_1 \text{ and } r_2 < r'_2) \text{ or } \ldots
\]
Then by Lemma~\ref{lem:estimate} we have $p(f_{\alpha}, L, z) > p(f_{\alpha}, L, z')$ for large enough $\alpha$ if and only if $z > z'$.

\smallskip

We apply this criterion to the deep holes of the Leech lattice. The deep holes of the Leech lattice were classified by Conway, Parker, Sloane \cite[Chapter 23]{Conway1988a}. There are $23$ inequivalent deep holes of $\Lambda_{24}$ which can be determined with the help of extended Coxeter Dynkin diagrams. In Table~\ref{table:Lambda24-deep-holes} we list all necessary information about them. From this table one sees that for large $\alpha$, cold spots are converging to deep holes of the type $A_{24}$. 
The inhomogenous shell $\Lambda_{24}(A_{24}, \sqrt{52/25})$ does not form a spherical $1$-design and so Theorem~\ref{thm:stable-critical-points} implies that $\Lambda_{24}$ does not have stable cold spots, which proves Theorem~\ref{thm:leech-lattice}.

\begin{table}
\begin{tabular}{lll}
\hline
deep hole & $(r_1, a_1, r_2, \ldots)$ & strength\\
\hline
\rule{0pt}{4mm}%
$A_{24}$ & $(\sqrt{2}, 25, \sqrt{52/25}, \ldots)$ & $1$ \\
$D_{24}$ & $(\sqrt{2}, 25, \sqrt{47/23}, \ldots)$ & $0$ \\

$A^2_{12}$ & $(\sqrt{2}, 26, \sqrt{28/13}, \ldots)$ & $2$ \\
$A_{15} D_9$ & $(\sqrt{2}, 26, \sqrt{17/8}, \ldots)$ & $0$ \\
$A_{17} E_7$ & $(\sqrt{2}, 26, \sqrt{19/9}, \ldots)$ & $0$ \\
$D^2_{12}$ & $(\sqrt{2}, 26, \sqrt{23/11}, \ldots)$ & $0$ \\
$D_{16} E_8$ & $(\sqrt{2}, 26, \sqrt{31/15}, \ldots)$ & $0$ \\

$A^2_9 D_6$ & $(\sqrt{2}, 27, \sqrt{11/5},  \ldots)$ & $0$ \\
$A^3_8$ & $(\sqrt{2}, 27, \sqrt{20/9}, \ldots)$ & $1$ \\
$A_{11} D_7 E_6$ & $(\sqrt{2}, 27, \sqrt{13/6}, \ldots)$ & $0$ \\
$D^3_8$ & $(\sqrt{2}, 27, \sqrt{15/7}, \ldots)$ & $0$ \\
$E^3_8$ & $(\sqrt{2}, 27, \sqrt{32/15}, \ldots)$ & $0$ \\
$D_{10} E^2_7$ & $(\sqrt{2}, 27, \sqrt{19/9}, \ldots)$ & $0$ \\

$A^4_6$ & $(\sqrt{2}, 28, \sqrt{16/7}, \ldots)$ & $1$ \\
$A^2_7 D^2_5$ & $(\sqrt{2}, 28, \sqrt{9/4}, \ldots)$ & $0$ \\
$D^4_6$ & $(\sqrt{2}, 28, \sqrt{11/5}, \ldots)$ & $0$ \\
$E^4_6$ & $(\sqrt{2}, 28, \sqrt{13/6}, \ldots)$ & $0$ \\

$A^4_5 D_4$ & $(\sqrt{2}, 29, \sqrt{7/3}, \ldots)$ & $0$ \\

$A^6_4$ & $(\sqrt{2}, 30, \sqrt{12/5}, \ldots)$ & $1$ \\
$D^6_4$ & $(\sqrt{2}, 30, \sqrt{7/3}, \ldots)$ & $0$\\

$A^8_3$ & $(\sqrt{2}, 32, \sqrt{5/2}, \ldots)$ & $1$\\
$A^{12}_2$ & $(\sqrt{2}, 36, \sqrt{8/3}, \ldots)$ & $2$\\
$A^{24}_1$ & $(\sqrt{2}, 48, \sqrt{3}, \ldots)$ & $3$\\
\hline     
\end{tabular}
\medskip
\caption{The deep holes of the Leech lattice. The design strength refers to the design strength of the first inhomogeneous lattice shell $\Lambda_{24}(c,\sqrt{2})$ of the deep hole $c$.}
\label{table:Lambda24-deep-holes}
\end{table}

\section*{Acknowledgements}

The authors like to thank Markus Faulhuber, Stefan Steinerberger, Ottaviano Marzorati for various fruitful discussions. F.V. thanks Universit\'e Toulouse-Jean Jaur\`es for its hospitality during several productive stays.

\appendix

\section{Indecomposable root lattices} \label{sec:Appendix:root-lattices}

In this appendix we put together some well known facts on root lattices and their duals, their Voronoi cells and the fundamental simplex of the associated affine Weyl groups.
This presentation closely follows \cite[Ch. 4]{Conway1988a}.

Firstly, by \cite[Ch. 21, Theorem 5]{Conway1988a} the Voronoi cell of a root lattice can be obtained as the union of the images of fundamental simplex (of the associated affine Weyl group) under the action of the (linear) Weyl group. 
In particular all holes of a root lattice are vertices of the fundamental simplex.
Note that this does not imply that every vertex of the fundamental simplex is a hole, this is, trivially, false for the vertex $0$ (which is a lattice point), but it can also be wrong for vertices distinct from $0$ (for example in the root systems of type $E$).

For an indecomposable root lattice $L$ we will provide the follwing data
\begin{enumerate}
  \item a standard representation of $L$,
  \item the dual lattice $L^*$ and the quotient $L^* / L$ to identify the $L$-translation classes of points to which Lemma \ref{lem:special:points:rootlattice} is applicable.
\end{enumerate}

\subsection{The $A_n$ series}

For $n \geq 1$ the root system $A_n$ generates the $n$-dimensional root lattice
\[
  A_n = \lset x \in \Z^{n+1} : e \cdot x = 1 \rset \subseteq \R^{n+1}.
\]
The quotient $A_n^* / A_n$ is cyclic of order $n+1$. A typical set of representatives for the elements of this quotient is given by the vectors
\[
  [i] = ( \underbrace{\tfrac{i}{n+1},\cdots,\tfrac{i}{n+1}}_{n+1-i \text{ times}},\underbrace{-\tfrac{n+1 -i}{n+1},\cdots,-\tfrac{n+1 - i}{n+1}}_{i \text{ times}}) \in A_n^*,
\]
where $i \in \lset 1,\ldots,n+1 \rset$. 
The vectors $[0],[1],\ldots,[n]$ are the vertices of a $n$-dimensional simplex and in fact that simplex is a fundamental simplex for the action of $W(A_n)^a$. 
Applying Lemma \ref{lem:special:points:rootlattice} we have that up to translation by vectors in $A_n$ these compose the set of special points for $W^a(A_n)$.

The vectors $[1],\ldots,[n]$ are all holes of $A_n$ (and each hole of $A_n$ is one of these up to a translation by some lattice element).
Of particular interest is the fact that up to the action of $W^a(A_n)$ there is only one deep hole in $A_n$, it is given by the orbit of $[a]$, where $a = \lfloor \tfrac{n+1}{2} \rfloor$.

\subsection{The $D_n$ series}

For $n \geq 3$ the root system $D_n$ generates the $n$-dimensional root lattice
\[
  D_n = \lset x \in \Z^n : e \cdot x \equiv_2 0 \rset.
\]
Note that $D_3 \cong A_3$.

The quotient $D_n^* / D_n$ is cyclic of order $4$ if $n$ is odd and noncyclic if order $4$ of $n$ is even. A typical set of representatives is given by the elements
\begin{align*}
  [0] &= (0,0,\ldots,0) \text{ of squared norm } 0\\
  [1] &= (\tfrac{1}{2},\tfrac{1}{2},\ldots,\tfrac{1}{2}) \text{ of squared norm } n/4\\
  [2] &= (0,0,\ldots,1) \text{ of squared norm } 1\\
  [3] &= (\tfrac{1}{2},\tfrac{1}{2},\ldots,-\tfrac{1}{2}) \text{ of squared norm } n/4.
\end{align*}
For $n=3$ there are two types of holes, given by the orbits of $[1]$ and $[3]$ (shallow) and the orbit of $[2]$ (deep).
For $n=4$ there is only one type of hole, given by the orbits of $[1]$, $[2]$, and $[3]$. 
For $n > 4$ there are two types of holes, given by the orbit of $[2]$ (shallow) and the orbits of $[1]$ and $[3]$ (deep).
Applying Lemma \ref{lem:special:points:rootlattice} we have that the set of special points for $W^a(D_n)$ consists of the lattice points $D_n$ and all holes.

For $n \geq 5$ the orthogonal group $\ort(D_n)$ is generated by $W(D_n) \cong \lset \pm 1 \rset^m_+ \rtimes S_n \cong \lset \pm 1 \rset^{n-1} \rtimes S_n$, where $\lset \pm 1 \rset^n_+$ operates by an even number of sign changes on the coordinates, and additionally the sign flip on the last coordinate (this symmetry interchanges $[1]$ and $[3]$).

The case $n =4$ is special.
In this case there is an additional symmetry given explicitely by the Hadamard matrix 
\begin{equation*}
  H_4 = \frac{1}{2} \begin{pmatrix}
    1 & 1 & 1 & 1\\
    1 & -1 & 1 & -1\\
    1 & 1 & -1 & -1\\
    1 & -1 & -1 & 1\\
  \end{pmatrix}.
\end{equation*}
Now the full orthogonal group is generated by $W(D_4)$, the sign flip on the last coordinate and the Hadamard symmetry. We denote by $\ort'(D_4)$ the subgroup only generated by $W(D_4)$ and the sign flip on the last coordinate. Then $[\ort(D_4):\ort'(D_4)] = 4$ and $[\ort'(D_4):W(D_4)]  = 2$.

\subsection{The $E_n$ series}

The $E_n$ series consists of the three lattices $E_6$ ,$E_7$, and $E_8$.

\subsubsection{The lattice $E_8$:} A common description of $E_8$ is in terms of the $D_n^+$ series and reads
\begin{align*}
  E_8 &= D_8^+ = D_8 \cup \tfrac{1}{2}e + D_8 \\
  &= \lset x \in \R^8 : e \cdot x \equiv_2 0, \text{ so that } x\in \Z^8 \text{ or } x \in \tfrac{1}{2}e + \Z^8 \rset.
\end{align*}
This defines an even unimodular $n$-dimensional lattice. 

There are two types of holes in $E_8$, given by the orbit of 
\[
  (0,0,0,0,0,0,0,1) \text{ of squared norm } 1
\]
(deep) and the orbit of
\[
  (\frac{1}{6},\frac{1}{6},\frac{1}{6},\frac{1}{6},\frac{1}{6},\frac{1}{6},\frac{1}{6},\frac{5}{6}) \text{ of squared norm } 8/9
\]
(shallow). It will turn out (by direct computation) that the deep holes of $E_8$ still are somehow ``special'', with stabilizer isomorphic to $W(D_8)$ (c.f. Lemma \ref{lem:deephole:stabilizer} \eqref{lem:deephole:stabilizer:E}).

\subsubsection{The lattice $E_7$:}
The lattice $E_7$ can be obtained as a $7$-dimensional sublattice of $E_8$, one common such description is 
\[
  E_7 = \lset x \in E_8 : v \cdot x = 0 \rset,
\]
where $v$ is any of the minimal vectors of $E_8$.
For explicit computations we describe $E_8$ as before and choose the minimal vector $v = \tfrac{1}{2} e$. Then 
\[
  E_7 = \lset x \in E_8 : e \cdot x = 0 \rset.
\]

The quotient $E_7^* / E_7$ is cyclic of order $2$. A typical set of representatives is given by the elements
\begin{align*}
  [0] &= (0,0,0,0,0,0,0,0) \text{ of squared norm } 0\\
  [1] &= (\tfrac{1}{4},\tfrac{1}{4},\tfrac{1}{4},\tfrac{1}{4},\tfrac{1}{4},\tfrac{1}{4},-\tfrac{3}{4},-\tfrac{3}{4}) \text{ of squared norm } 3/2.
\end{align*} 
The orbit of $[1]$ gives all deep holes of $E_7$ and this is the only type of hole.
Applying Lemma \ref{lem:special:points:rootlattice} we have that the set of special points for $W^a(E_7)$ consists of the lattice points $E_7$ and all holes.

\subsubsection{The lattice $E_6$:}
The lattice $E_6$ can be obtained as a $6$-dimensional sublattice of $E_8$, one common such description is 
\[
  E_6 = \lset x \in E_8 : v \cdot x = 0 \text{ for all } v \in V \rset,
\]
where $V$ is any sublattice of $E_8$ isomorphic to $A_2$.
For explicit computations we describe $E_8$ as above and choose the sublattice $\langle e_1+e_8, \tfrac{1}{2}e \rangle_\Z \cong A_2$. Then 
\[
  E_6 = \lset x \in E_8 : e \cdot x = 0 \rset.
\]

The quotient $E_6^* / E_6$ is cyclic of order $3$. A typical set of representatives is given by the elements
\begin{align*}
  [0] &= (0,0,0,0,0,0,0,0) \text{ of squared norm } 0\\
  [1] &= (0,-\tfrac{2}{3},-\tfrac{2}{3},\tfrac{1}{3},\tfrac{1}{3},\tfrac{1}{3},\tfrac{1}{3},0) \text{ of squared norm } 4/3\\
  [2] &= -[1] \text{ of squared norm } 4/3.
\end{align*} 
The orbit of $[1]$ gives all deep holes of $E_6$ and this is the only type of hole.
Applying Lemma \ref{lem:special:points:rootlattice} we have that the set of special points for $W^a(E_6)$ consists of the lattice points $E_6$ and all holes.

\section{Some background on Weyl groups and root systems} \label{sec:Appendix:weyl}
We base the subsequent treatise on affine Weyl groups on \cite{Bourbaki1968}. We use slightly adapted notation and base all objects in the same vector space, instead of using the dual space, as would be common in the study of root systems. 
Let $R$ be a (reduced) root system in the real vector space $\R^n$.
To $R$ we can associate another root system, the \emph{inverse root system} (or system of coroots) $R^{\vee}$ comprising the elements $\alpha^\vee = \tfrac{2}{\|\alpha\|^2} \alpha$.
These sets induce lattices, the \emph{root lattice} 
\[
  L_R = \langle R \rangle_\Z 
\]
the \emph{coroot lattice}
\[
  L_{R^{\vee}} = \langle R^\vee \rangle_\Z
\]
and the \emph{weight lattice}
\[
  L_{R^{\vee}}^* = \lset x \in \R^n : x \cdot v \in \Z \text{ for all } v \in L_{R^\vee} \rset.
\]
which is the dual lattice of $L_{R^\vee}$ in the usual sense. 

Note that in the language of root systems the elements of $L_R$ are called the radical weights of the root system $R$, while the elements of the weight lattice $L_{R^\vee}^*$ are called weights.
To a root system we can associate its Weyl group 
\[
  W(R) = \langle s_\alpha : \alpha \in R \rangle,
\]
where $s_\alpha$ is the reflection along $\alpha$, that is $s_\alpha(\alpha) = -\alpha$ and $s_\alpha(x) = x$ for $x \in \alpha ^\perp$.
The associated affine Weyl group is
\[
  W(R)^a = W(R) \rtimes T(L_{R^{\vee}}),
\]
where $T(L_{R^{\vee}})$ is the set of translations induced by elements of $L_R^{\vee}$ (see \cite[Ch. VI, \S 2, no. 1.]{Bourbaki1968}).

In the context of an affine Weyl group $W^a$ a point $z \in \R^n$ is called \emph{special} for $W^a$ if and only if it satisfies a number of equivalent conditions, in particular if and only if $W_z^a \cong W(R)$ (c.f. \cite[Ch. V, \S 3, no. 10.]{Bourbaki1968}).


By \cite[Ch. VI, \S 2, Proposition 3]{Bourbaki1968} the special points of the affine Weyl group $W^a(R)$ of a root system $R$ are precisely the elements of the weight lattice of $R^\vee$, which is the lattice $L_R^*$.

We finally note that if $L$ is root lattice as in \eqref{eq:root:lattice} the underlying root system $R = L(2)$ is self dual, that is $R^\vee = R$, and all of the considerations above become a bit simpler because already $L = L_R = L_{R^\vee}$.

\end{document}